\documentclass[11pt]{amsart}
\pdfoutput=1
\usepackage{comment,todonotes}

\usepackage[T1]{fontenc}
\newfont{\script}{rsfs10 scaled 1200}
\newtheorem{thm}{Theorem}
\newtheorem{corollary}{Corollary}
\newtheorem{lem}{Lemma}[section]

\theoremstyle{remark}

\newcommand{\bc}{\mathbf{c}}
\theoremstyle{definition}
\newtheorem{defn}{Definition}[section]

\title{The  Inhomogeneous Hall's  Ray} 
\author{D.J. Crisp, W. Moran and A.D. Pollington}
\date{March   2012}

\begin{document}
\maketitle
\section{Introduction}
\label{sec:1www}
The expression 
\begin{equation*}
{\mathcal M}_{+}(\alpha,\beta)=\liminf_{q\to\infty}q||q\alpha-\beta||
\end{equation*}
measures how well multiples of a fixed irrational $\alpha>0$
approximate a real number $\beta$. A similar concept is defined by Rockett
and Sz\"usz (\cite{rockett92:_contin_fract} Ch. 4, \S9), where they consider
the slight variant, ${\mathcal M}(\alpha,\beta)$, (the \emph{two-sided} case) with the
initial $q$ replaced by $|q|$.  It is evident that (see, for example,
\cite{pinnner01:Moron,Komatsu1997192})
\begin{equation*}
  {\mathcal M}(\alpha,\beta)=\min\bigl({\mathcal M}_{+}(\alpha,\beta),{\mathcal M}_{+}(\alpha,-\beta)\bigr). 
\end{equation*}
We define
\begin{equation}
  \label{eq:4}
  \begin{aligned}[t]
  {\mathcal S}_{+}(\alpha)&=\{{\mathcal M}_{+}(\alpha,\beta): \beta\in {\mathbf R}^{+}\}\\
  {\mathcal S}(\alpha)&=\{{\mathcal M}(\alpha,\beta): \beta\in {\mathbf R}^{+}\}. 
  \end{aligned}
\end{equation}
We refer to the first set  as the \emph{(one-sided) inhomogeneous
  approximation spectrum} of $\alpha$.

${\mathcal M}_{+}(\alpha,\beta)$ and the corresponding spectrum have been considered in
precisely this form by various authors,
\cite{komatsu99:_inhom_dioph_approx_some_quasi_period_expres,komatsu99:_ii,cusick96:_halls_ray_in_inhom_dioph_approx,blanksby67:_various_probl_inhom_dioph_approx},
and the ideas relate to inhomogeneous minima of binary quadratic forms
\cite{barnes56:_linear_inhom_dioph_approx,blanksby67:_various_probl_inhom_dioph_approx,MR0053162,MR0054654,MR0067939,barnes54,barnes2}.
 In the celebrated paper
(\cite{m.47:_sum_and_produc_of_contin_fract}), Hall showed that
the \emph{Lagrange spectrum},
${\mathcal L}=\{{\mathcal M}_{+}(\alpha,0):\alpha\in {\mathbf R}\}$,
 contains an interval $[0,\mu_{H}]$ ($\mu_H>0$) subsequently called \emph{Hall's
  Ray}. The precise value of $\mu_{H}$ has been determined by Freiman
(\cite{freiman73:_hall}) in a heroic calculation; we refer the reader to
\cite{cusick89:_markof_lagran}, where this result is discussed in detail.  Our
aim here is to prove  the existence of an interval $[0,\mu_{\alpha}]$ in the inhomogeneous spectrum for
all irrationals $\alpha$, though without a precise value for the maximum
endpoint of the interval. It is clear that the result fails for rational
$\alpha$.

Since ${\mathcal M}_{+}(\alpha,\beta)={\mathcal M}_{+}(\alpha,\beta+1)$, the values of $\beta$
may be restricted to the unit interval $[0,1)$.  Similarly, we may assume
without loss of generality that $0\le\alpha<1$.  The key theorem of this paper
is the following:
\begin{thm}
\label{thm:1}
For $\alpha$ irrational, the set ${\mathcal S}_{+}(\alpha)$ contains
an interval of the form $[0,\mu_{\alpha}]$ for some $\mu_{\alpha}>0$.  
\end{thm}
Once this is established, it is straightforward to extend to the two-sided
case, and to binary quadratic forms. 

\subsection{History}

As far as we are aware, the first work  on inhomogeneous minima
dates back to Minkowski \cite{minkowski01:_ueber_annaeh_groes_zahlen}
who expressed his results in terms of  binary quadratic forms. He  showed  that
if $a,b,c,d$ are real numbers with $\Delta=ad-bc\neq 0$ then, for any
real numbers $\lambda$ and $\mu$, there are integers $m,n$ such that
\begin{equation*}
  |(am-bn-\lambda)(cm-dn-\mu)|\leq \frac{1}{4}\Delta.
\end{equation*}
This implies that $\inf_q|q|||q\alpha-\beta||\leq \frac{1}{4}$ for all
$\alpha,\beta$.  The same conclusion is true for
$\mathcal M(\alpha,\beta)$ but this requires more work. In
fact Khintchine \cite{khintchine35:_neuer_beweis_veral_hurwit_satzes}
proved that $\mathcal M_+(\alpha,\beta)\leq \frac{1}{3}$, and 
the result with $\frac{1}{4}$ replacing $\frac{1}{3}$
is claimed by Cassels as derivable from his methods in
\cite{cassels54:ueber}. 

Khintchine \cite{khintchine26:_ueber_klass_approx} showed that there
exists $\delta>0$ such that, for any $\alpha$, there exists $\beta$
for which
\begin{equation*}
  \mathcal M(\alpha,\beta)\geq \delta.
\end{equation*}
In fact, like Minkowski, he deals with the infimum rather than
$\liminf$.  Fukusawa gave an explicit value for $\delta$ of $1/457$
and this was subsequently improved by Davenport ($\delta=1/73.9$)
\cite{davenport50:_theor_khint} and
by  Prasad ($\delta=3/32$)~\cite{prasad51}.  These papers are of special
significance  because they develop a methodology for handling
calculations of values of $\mathcal M_+(\alpha,\beta)$ that has been
the cornerstone of   much subsequent work,  and underlies the
techniques  used in this paper. 

Far too many authors have contributed to the understanding of
${\mathcal M}(\alpha,\beta)$ and $\mathcal M_+(\alpha,\beta)$ for us
to reference all of the papers here.  As far as we are aware, the
first ray results occur in
\cite{fukasawa26:_ueber_groes_betrag_form_i}, Satz XIII, where it is
shown that if, in a semi-regular continued fraction expansion of
$\alpha$, the partial quotients tend to $\infty$ then ${\mathcal
  S}(\alpha)$ contains the interval $[0,\frac{1}{4}]$. Barnes obtains
essentially the same result in
\cite{barnes56:_linear_inhom_dioph_approx} though he states a weaker
one: that, for each $t\in [0,\frac{1}{4}]$, there are uncountably many
$\alpha$'s and $\beta$'s with ${\mathcal M}(\alpha,\beta)=t$.

The predominant methodology for handling problems of this kind, originating with Davenport
\cite{davenport50:_theor_khint},   invokes
some  form of continued fraction expansion of $\alpha$ and a corresponding
digit expansion of $\beta$. We will use this methodology but
choose to use the negative continued fraction
because of the simple and ``decimal''-like geometrical interpretation of the
expansion of $\beta$ associated with it (which we call the
\emph{Davenport  Expansion}). Use of the regular continued fraction is possible,  and
was first  done by Prasad~\cite{prasad51}, but makes the  construction less
intuitive and more complicated  from our perspective,   because divisions of subintervals alternate in
direction.  The general machinery for the regular continued fraction is
well-exposed in Rockett and Sz\"usz \cite{rockett92:_contin_fract}.
Cassels also uses the Davenport expansion ideas in his paper \cite{cassels54:ueber},
without attribution, where he shows that,  except for special cases, 
$\mathcal M_+(\alpha,\beta)\leq \frac{4}{11}$.  Several authors have
contributed to refinement of the technique, including S\'os
\cite{sos58:_dioph_approx_ii}, and Cusick, Rockett and Sz\"usz
\cite{cusick94:_inhom_dioph_approx}.  These authors ascribe the origin
of the technique to Cassels in \cite{cassels54:ueber}. 

Almost all of the work for this paper, including a more complicated
proof of the main theorem,  was done in the early 1990's, and
versions of it have been circulating privately since then. Its ideas and
results have been used and cited in various places, in particular, in
\cite{pinnner01:Moron,pinner01:_inhom_halls_ray_period}.

\section {The negative continued fraction expansion for $\alpha$}
\label{sec:23456}
Here we briefly describe the features needed from the theory of  the negative continued
fraction.  For a more complete discussion of the corresponding concepts for
the regular continued fraction, see \cite{rockett92:_contin_fract} or for the
more general semi-regular continued fraction see
Perron~\cite{perron54:_lehre_ketten}.  For $0<\alpha<1$,
let $\alpha_1=\alpha$, $a_1=\lceil\frac1{\alpha_1}\rceil$, and define,
recursively,
\begin{equation*}
a_i=\left\lceil\frac1{\alpha_i}\right\rceil\qquad
  \text{and}\qquad\alpha_{i+1}=a_i-\frac1{\alpha_i}. 
\end{equation*}
so that $a_i\ge 2$ and $0<\alpha_{i+1}<1$, for all $i$.
Evidently,  $\alpha$ has the continued fraction expansion
\begin{equation*}
  \label{eq:2}
\alpha=
\cfrac{1}{a_1-\cfrac{1}{a_2-\cfrac1{a_3-\cfrac{1}{\ddots}}}},
\end{equation*}
abbreviated as $\alpha=\langle a_1,a_2,a_3,\ldots\rangle$.  The numbers
$\alpha_i$ are called the $i$th \emph{complete quotients} of $\alpha$ and
satisfy
\begin{equation*}
  \label{eq:7}
\alpha_i=\langle a_i,a_{i+1},a_{i+2},\ldots\rangle.  
\end{equation*}
Since $\alpha$ is irrational, the \emph{partial quotients} $a_i$ are greater
than $2$ for infinitely many indices $i$, and so there  is a unique
sequence $a'_1,a'_2,a'_3,\ldots$ of positive integers such that
\begin{equation}
  \label{eq:8}
  a_1,a_2,a_3,\ldots=a'_1+1,\underbrace{2,\ldots,2}_{a'_2-1},
  a'_3+2,\underbrace{2,\ldots,2}_{a'_4-1},a'_5+2,\underbrace{2,\ldots,2}_{a'_6-1},
  a'_7+2,\ldots.
\end{equation}
It will be necessary occasionally to discuss the usual continued fraction
expansion of $\alpha$,  now expressible as
\begin{equation}
  \label{eq:9}
  \alpha=\cfrac{1}{a'_1+\cfrac{1}{a'_2+\cfrac{1}{a'_3+\cfrac{1}{\ddots}}}}.
\end{equation}
Eventually, we will split the proof of Theorem~\ref{thm:1} into two
cases, corresponding to whether or not the sequence  $(a_{n}')$ is bounded.  

We make use of the (negative continued
fraction) \emph{convergents} $p_i/q_i$ to $\alpha$:
\begin{equation}
  \label{eq:18}
  \frac{p_i}{q_i}=\langle a_1,a_2,\ldots,a_i\rangle,
\end{equation}
 satisfying the recurrence relations
\begin{equation}
  \label{eq:19}
  p_{i+1}=a_{i+1}p_i-p_{i-1}\qquad\text{and}\qquad
  q_{i+1}=a_{i+1}q_i-q_{i-1}
\end{equation}
where $i\ge1$ and $p_0,\ q_0=0,\ 1$. Easily established are the
following simple properties: 
\begin{align}
  \label{eq:21}
  1&=p_iq_{i-1}-q_ip_{i-1}\\
\label{eq:22}  
\alpha&=\frac{(a_i-\alpha_{i+1})p_{i-1}-p_{i-2}}{(a_i-\alpha_{i+1})q_{i-1}-q_{n-2}}
  =\frac{p_i-\alpha_{i+1}p_{i-1}}{q_i-\alpha_{i+1}q_{i-1}}.
\end{align}
Moreover, $q_{i-1}/q_i=\overline\alpha_i$ where
\begin{equation}
  \label{eq:24}
  \overline\alpha_i=\langle a_i,a_{i-1},\ldots,a_1\rangle.
\end{equation}
Since $q_0=1$,  the identity
\begin{equation}
  \label{eq:25}
  q_i=\frac1{\overline\alpha_1\overline\alpha_2\ldots\overline\alpha_i}
\end{equation}
follows. 

This section concludes with a brief description of  the \emph{Ostrowski expansion} (see
\cite{rockett92:_contin_fract}) 
for  positive integers.  Any given
integer $q\ge1$ can be written as a sum of the form
\begin{equation}
  \label{eq:27}
  q=\sum^n_{k=1}c_kq_{k-1}
\end{equation}
where
\begin{equation}
  \label{eq:28}
  c_n\ge1\qquad\text{and}\qquad 0\le c_k\le a_k-1\qquad
  \text{for}\qquad1\le k\le n.
\end{equation}
A greedy algorithm is used to determine the coefficients
$c_{n}$.

It is not hard to verify that
\begin{equation}
  \label{eq:31}
  q_k-1=(a_1-2)q_0+(a_2-2)q_1+\cdots
  +(a_{k-1}-2)q_{k-2}+(a_k-1)q_{k-1}.
\end{equation}
This last identity yields that, for
 no pair of indices $i$ and $j$, is 
there a consecutive subsequence  of coefficients of the form
\begin{equation}
  \label{eq:32}
(c_i,c_{i+1},\dots,c_j) =( a_i-1,a_{i+1}-2,a_{i+2}-2,\ldots,a_{j-1}-2,a_j-1).
\end{equation}
The basic facts about the Ostrowski expansion are described in the
following lemma.   
\begin{lem}
\label{thm:98983}
Each integer $q\ge1$ has a unique expansion of the form \eqref{eq:27}
such that the constraint \eqref{eq:28} holds and
no consecutive sub-sequence of coefficients  is of the form \eqref{eq:32}.
\end{lem}

\section{ The Davenport expansion of $\beta$}
\label{sec:3gfd}
We now describe  the \emph{Davenport Expansion} for the elements
$\beta$ of the interval $[0,1)$.  While the expansion is analogous to that
used in \cite{cusick96:_halls_ray_in_inhom_dioph_approx}, we remind the reader
that it is based on a different continued fraction algorithm.  This approach
results in a ``decimal''-like geometry of the Davenport expansion in the
negative continued fraction case which makes  more intuitive  the invocation of Hall's
theorem on sums of Cantor
sets~\cite{m.47:_sum_and_produc_of_contin_fract} later. This is a key
component of the proof in the bounded case. 

For  $0\le\beta<1$, 
let $\beta_1=\beta$ and define, inductively,
\begin{equation*}
  \label{eq:42}
  b_i=\left\lfloor\frac{\beta_i}{\alpha_i}\right\rfloor\qquad\text{and}\qquad
  \beta_{i+1}=\frac{\beta_i}{\alpha_i}-b_i.
\end{equation*}
so that $0\le b_i\le a_i-1$ and $0\le\beta_{i+1}<1$.
The convergent sum $\beta=\sum_{k=1}^\infty b_kD_k$  is called 
the \emph{ Davenport expansion} of $\beta$ or the \emph{Davenport sum}
of the sequence $(b_k)$ relative to $\alpha$. The
integers $b_{i}$ are the \emph{Davenport coefficients}. In the same way as
in the decimal expansion $0.999\ldots$ is identified with
$1.000\ldots$, we  identify
\begin{equation}
  \label{eq:decimal_ambig}
  b_1, b_2, \ldots, b_i, a_i-1, a_{i+1}-2, a_{i+2}-2, \ldots, \text{
    with }  b_1, b_2, \ldots, b_i+1, 0,0, \ldots
\end{equation}
for $b_i<a_i-1$, since their Davenport sums are the same. 

Figure~\ref{fig:1} gives an illustration of the geometry of the
situation for the case when $\alpha=\langle
5,3,5,3,\ldots\rangle$.
The interval $[0,1)$ is
subdivided by the numbers $n\alpha \pmod 1$, $(n=1,2,3, 4)$ into $5$ intervals,
the first four of which are ``long'' and the last  ``short'' since $5\alpha>1$. When
we allow $n$ to range up to $13$, each long interval is then subdivided into $3$
intervals with the same pattern in each: $2$ ``long'' intervals and $1$ ``short''
interval, whereas the ``short'' interval is divided into just $1$ ``long''
interval and $1$ ``short'' interval.  This pattern of ``long'' and ``short''
intervals is repeated at finer and finer resolutions as $n$ increases,
reflecting, in this example,  the periodic structure of the continued fraction.  This
structure corresponds to a ``decimal'' expansion with restrictions on
digits, involving
dependencies on the preceding digits. The general case is described below.

\begin{figure}[ht!]
  \centering
\includegraphics[width=\textwidth]{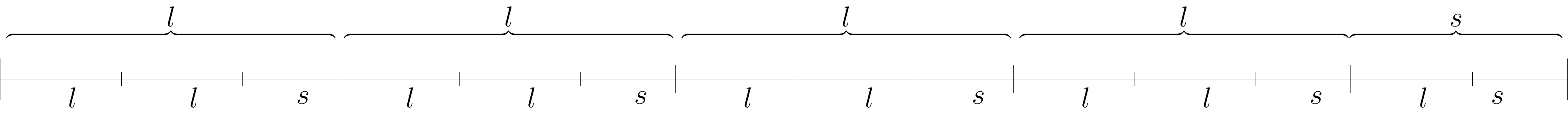}
  \caption{The ``Long-Short'' Picture for $\alpha=\langle 5,3,5,3,\dots\rangle$}
  \label{fig:1}
\end{figure}

From the inductive step in the Davenport expansion, 
\begin{equation*}
  \label{eq:43}
  \beta_i=b_i\alpha_i+\beta_{i+1}\alpha_i
\end{equation*}
and, as a result, 
\begin{equation}
  \label{eq:44}
  \beta_i=b_i\alpha_i+b_{i+1}\alpha_i\alpha_{i+1}+\ldots
  +b_j(\alpha_i\alpha_{i+1}\ldots\alpha_j)
  +\beta_{j+1}(\alpha_i\alpha_{i+1}\ldots\alpha_j)
\end{equation}
for all $j\ge i$. 
Note that $\beta_{i}$ is the location of $\beta$ in the rescaled copy
of the (long) interval in which it is contained.   
We define
\begin{equation*}
  \label{eq:45}
D_{1}=1,\qquad  D_i=\alpha_1\alpha_2\ldots\alpha_i
\end{equation*}
and write
\begin{equation*}
  \label{eq:46}
  \beta_iD_{i-1}=b_iD_i+b_{i+1}D_{i+1}+\cdots
  +b_jD_j+\beta_{j+1}D_j.
\end{equation*} $D_{i}$ is the length of the long intervals at the $i$th
level, and $D_{i}-D_{i+1}$ is the length of the short intervals at that level. 

The following result is straightforward. 
\begin{thm}
\label{thm:hdjd}
Let $\beta=\sum^\infty_{k=1}b_kD_k$ where
$(b_i)$ is a sequence of positive integers.
Then $0\le\beta<1$ and $(b_i)$ are  the Davenport coefficients of $\beta$
if and only if $b_i<a_i$ for all $i\ge1$ and 
no block of the  form
\begin{equation}
  \label{eq:69}
  a_i-1,a_{i+1}-2,a_{i+2}-2,\dots,a_{j-1}-2,a_j-1
\end{equation}
or of the form
\begin{equation}
  \label{eq:70}
  a_i-1,a_{i+1}-2,a_{i+2}-2,a_{i+3}-2,\ldots
\end{equation} occurs in $(b_i)$. 
\end{thm}
The exceptional cases in this result; when $b_i,b_{i+1},\ldots,b_j$ is
of the form $ a_i-1,a_{i+1}-2,a_{i+2}-2,\dots,a_{j-1}-2,a_j-1$,
correspond to the  missing long intervals in the short intervals one level
higher. As in the example in Figure~\ref{fig:1}, each short interval
has one fewer long interval at the next level. In the general
geometric picture, $a_{1}-1$ multiples of $\alpha$ subdivide the unit
interval into $a_{1}$ intervals, the first $a_{1}-1$ of which have
length $\alpha$ and the last of length $1-(a_{1}-1)\alpha$. The next
multiple (modulo 1) is $\alpha_{1}\alpha_{2}=a_{1}\alpha-1$. This
subdivides each of the long intervals at the previous level into
$a_{2}-2$ intervals of the same length followed by a short
interval. The final short interval of the initial subdivision is
subdivided into $a_{2}-2$ long intervals followed by a short
interval. This pattern is repeated at all finer resolutions with  the
appropriate partial quotients.

By means of the Davenport expansion, we can describe the integer pairs $(p,q)$
for which $0<q\alpha-p<1$.  It is straightforward to see that if
$q=\sum^n_{k=1}c_kq_{k-1}$ is the Ostrowski expansion of $q$ then
\begin{equation}
  \label{eq:72}
  p=\sum^n_{k=1}c_kp_{k-1},\quad i\ge 1.
\end{equation}

\begin{lem}
\label{thm:klksjd}
  \begin{enumerate}
\item Let $q\ge1$ be an integer with Ostrowski expansion as in
\eqref{eq:27} and let  
$p$ be defined by \eqref{eq:72}.  Then $0<q\alpha-p<1$ and
\begin{equation*}
  \label{eq:75}
q\alpha-p=\sum^\infty_{k=1}b_kD_k  
\end{equation*}
is the Davenport expansion of $q\alpha-p$,
where $(b_i)$ is the sequence
$c_1,c_2,\ldots,c_n,0,0,0,\ldots$.
  \item Let $0<\beta<1$ and let $(b_i)$ be the Davenport coefficients of $\beta$.
Then there are integers $q\ge1$ and $p$ such that $\beta=q\alpha-p$
if and only if there is $n\ge1$ such that $b_i=0$ for all $i>n$.
Further, if that is so then
$q=\sum^n_{k=1}b_kq_{k-1}$ and $p=\sum^n_{k=1}b_kp_{k-1}$.
  \end{enumerate}
\end{lem}

\section{Calculation of ${\mathcal M}_{+}(\alpha,\beta)$ via the Davenport Expansion}
\label{sec:4gfala}
The Davenport expansion will be  used to calculate
${\mathcal M}_{+}(\alpha,\beta)$. Again we stress that the underlying ideas are not
really new, being essentially contained in the work of Davenport,
Cassels, S\'os,  and
others. Accordingly, we omit much of the justification and instead aim to provide
geometrical insights.

To begin, let $0\le\beta<1$ and let $(b_i)$ be the Davenport
coefficients of $\beta$.  We define
\begin{equation*}
  \label{eq:95}
  \begin{aligned}[t]
  Q_n&=\sum^n_{k=1}b_kq_{k-1}\\
  Q'_n&=
  \begin{cases}Q_n+q_{n-1}&\text{ if $Q_n<q_n-q_{n-1}$}\\
  Q_n+q_{n-1}-q_n&\text{ if $Q_n\ge q_n-q_{n-1}$}
  \end{cases}
 \end{aligned}
\end{equation*}
for all $n\ge1$.  The two cases here correspond to when $\beta$
lies in a long or a short interval, respectively, at the appropriate level of
the decomposition of the interval. If $\beta$ is in a short interval, then the
right endpoint of that interval occurred earlier in the decomposition; hence
the $q_{n}-q_{n-1}$ term.

The next two lemmas are relatively straightforward consequences of these
definitions and ideas. 
\begin{lem}\label{thm:5543}
  \begin{enumerate}
\item
$0\le Q_n<q_n$ for all $n\ge1$ and
$Q_n\ge q_{n-1}$ if and only if $b_n\ne0$.
\item
$Q_n\ge Q_{n-1}$ for all $n\ge2$ and
$Q_{n-1}=Q_n$ if and only if $b_n=0$.
\item
$0\le Q'_n<q_n$ for all $n\ge1$ and
$Q'_n\ge q_{n-1}$ if and only if $Q_n<q_n-q_{n-1}$.
\item
$Q'_n\ge Q'_{n-1}$ for all $n\ge2$ and
$Q'_{n-1}=Q'_n$ if and only if $Q_n\ge q_n-q_{n-1}$.
\item The inequality $Q_n\ge q_n-q_{n-1}$ holds if and only if
there is some index $m$ with $1\le m\le n$ such that
the sequence $b_m,b_{m+1},\ldots,b_n$ is equal to
\begin{equation*}
  \label{eq:102}
  a_m-1,a_{m+1}-2,a_{m+2}-2,\ldots,a_n-2.
\end{equation*}
  \end{enumerate}
\end{lem}
The last condition, $Q_n\ge q_n-q_{n-1}$,
occurs if the point $\beta$ is inside  a
short interval. 

The integers $Q_n$ and $Q'_n$ are used to define quantities $\lambda_n(\beta)$
and $\rho_n(\beta)$, the significance of which will be evident from the
following lemma.
\begin{defn}
Let $0\le\beta<1$ and let $\beta_1,\beta_2,\beta_3,\ldots$ be the sequence
of numbers generated by applying the Davenport expansion algorithm to $\beta$.
We define
\begin{equation}
  \label{eq:103}
  \lambda_n(\beta)=Q_nD_n\beta_{n+1}
\end{equation}
and
\begin{equation}
  \label{eq:104}
  \rho_n(\beta)=
  \begin{cases}
  Q'_nD_n(1-\beta_{n+1})&\text{ if $Q_n<q_n-q_{n-1}$}\\    
  Q'_nD_n(1-\alpha_{n+1}-\beta_{n+1})&\text{ if $Q_n\ge q_n-q_{n-1}$}
  \end{cases}
\end{equation}
for all $n\ge1$.
\end{defn}
Recall that $Q_{n}$ is the ``count'' of $q\alpha$ that corresponds the left
endpoint of the interval at level $n$ that contains $\beta$, and that $D_{n}$
is the length of a long interval at that level. It follows that $D_n\beta_{n+1}$ is the
distance to $\beta$ from the left endpoint of the interval at level $n$
containing $\beta$. In similar vein, $\rho_{n}(\beta)$ is the count for the
right endpoint of that interval multiplied by the distance from $\beta$ to
that endpoint. The next lemma is straightforward from the geometrical
picture of the
interval decompositions.

\begin{lem}
\label{thm:6795}
Let $n<m$, $0<\beta<1$,
and $(b_i)$ be the Davenport coefficients of $\beta$, with
$b_i\ne 0$ for infinitely many $i$. Then
\begin{enumerate}
\item \begin{equation*}
  \label{eq:105}
  \lambda_n(\beta)=Q_n||Q_n\alpha-\beta||\qquad\text{and}\qquad
  \rho_n(\beta)=Q'_n||Q'_n\alpha-\beta||.
\end{equation*}
\item If $b_n=0$ then $\lambda_n(\beta)=\lambda_{n-1}(\beta)$. In other words
  if $\beta$ is in the first interval of the decomposition at level $n$, then
  $Q_{n}\alpha$ is $Q_{n+1}\alpha$ modulo 1.
\item
If $b_n\ne0$ and $b_m\ne0$ and
$b_i=0$ for all $i$ which satisfy $n<i<m$ then
$q_{n-1}D_m\le\lambda_n(\beta)<q_nD_{m-1}$.
\item If $Q_n\ge q_n-q_{n-1}$ then $\rho_n(\beta)=\rho_{n-1}(\beta)$. In other
  words, if $\beta$ is a short interval (namely a rightmost) at level $n$ then
  $Q_{n}'\alpha$ is equal to $Q_{n}\alpha$ modulo 1.
\item
If $Q_n<q_n-q_{n-1}$ and $Q_m<q_m-q_{m-1}$ and
$Q_i\ge q_i-q_{i-1}$ for all $i$ which satisfy $n<i<m$ then
$q_{n-1}D_m(1-\alpha_{m+1})\le\rho_n(\beta)<q_nD_{m-1}(1-\alpha_m)$
unless $m=n+1$ in which case
$q_{n-1}D_{n+1}(1-\alpha_{n+2})\le\rho_n(\beta)<q_nD_n$.
\end{enumerate}
\end{lem}

The next lemma is a key step in calculating ${\mathcal M}_{+}(\alpha,\beta)$ in terms of
$\lambda_n(\beta)$ and $\rho_n(\beta)$. 
\begin{lem}
\label{thm:gda}
For $n\geq 1$,
\begin{equation}
  \label{eq:112}
  \min\{\lambda_n(\beta),\rho_n(\beta),\lambda_{n+1}(\beta),\rho_{n+1}(\beta)\}
\end{equation}
is a lower bound for the infimum of the set
$\{q||q\alpha-\beta||:\;q_n\le q<q_{n+1}\}$.
\end{lem}

\begin{proof} We sketch the proof of the result. 
The diagram showing the key ideas is given in Figure~\ref{fig:2}. 
\begin{figure}[ht!]
  \includegraphics[width=0.9\textwidth]{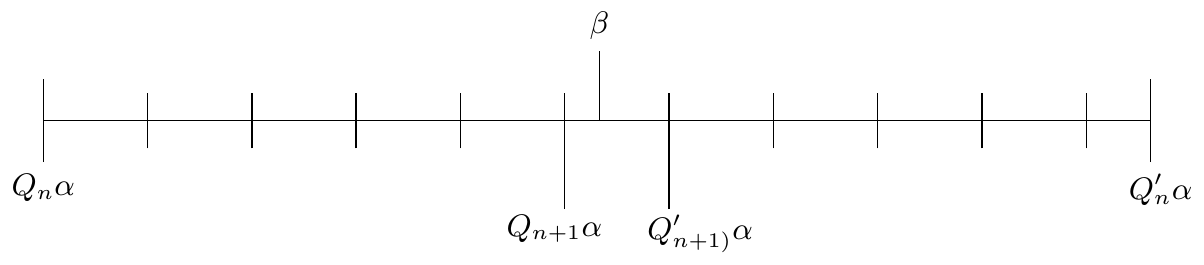}
  \caption{The Approximations of $\beta$ }
  \label{fig:2}
\end{figure}
Write $I_{n}$ and $I_{n+1}$ for the intervals prescribed by the
Davenport expansion at level $n$ and $n+1$ that contain $\beta$:
$I_{n}=[Q_{n}\alpha,Q_{n}'\alpha]$,
$I_{n+1}=[Q_{n+1}\alpha,Q_{n+1}'\alpha]$.  The obvious candidates for
the smallest values of $q||q\alpha-\beta||$ for $q_{n}\leq q\leq
q_{n+1}$ are the cases $q=Q_{n+1}$ or $q=Q_{n+1}'$ --- the left and
right endpoints of the interval $I_{n+1}$ at level $n+1$ containing
$\beta$. It is clear from fairly straightforward size considerations
that they do better than any $q\alpha \in I_{n}\ (q_{n} \leq q\leq
q_{n+1})$. It is also clear that the candidates $q=Q_{n}$ and
$q=Q'_{n}$ are better than any $q\alpha \not\in I_{n} (q_{n} \leq
q\leq q_{n+1})$ since $Q_{n}<q_{n}$.
\end{proof}

The key  equation for calculation of ${\mathcal M}_{+}(\alpha,\beta)$
is in the following theorem, which captures the important ingredient of the
preceding lemma.
\begin{thm}
\label{thm:333}
If $0<\beta<1$ and  no integers $q\ge1$ and $p$ satisfy
$\beta=q\alpha-p$ then
\begin{equation}
  \label{eq:143}
  {\mathcal M}_{+}(\alpha,\beta)=\min\left\{\liminf_{n\to\infty}\lambda_n(\beta),\;
  \liminf_{n\to\infty}\rho_n(\beta)\right\}.
\end{equation}
\end{thm}

For completeness, we note that, in Theorem~\ref{thm:333}, we have not dealt with
the possibility that $\beta$ is of the form $q\alpha-p$ where $q$ and $p$ are
positive integers.  In this case, we have
\begin{equation}
  \label{eq:151}
  {\mathcal M}_{+}(\alpha,\beta)=\liminf_{q'\to\infty} q'||q'\alpha-q\alpha-p||
  =\liminf_{q'\to\infty} q'||(q'-q)\alpha||
\end{equation}
and consequently
\begin{equation}
  \label{eq:152}
  {\mathcal M}_{+}(\alpha,\beta)=\liminf_{q'\to\infty} (q'-q)||(q'-q)\alpha||={\mathcal M}_{+}(\alpha,0).
\end{equation}
The quantity ${\mathcal M}_{+}(\alpha,0)$ it is, of course, the homogeneous
approximation constant of $\alpha$.

\section{The Unbounded Case}
\label{sec:5adas}
In this section we dispense quickly and relatively straightforwardly with the
case where $\alpha$ has unbounded partial quotients ($a^\sharp_{n}$) in its ordinary
continued fraction, before turning to the much more difficult case of bounded
partial quotients.  We write
\begin{equation*}
  \label{eq:6}
\mathcal{M}_+ (\alpha)=\sup_{\beta} {\mathcal M}_{+}(\alpha,\beta).
\end{equation*}

The following theorem is the key result of this section. 
\begin{thm}
  \label{thm:unbounded}
  If $\alpha$ has unbounded partial quotients in its ordinary continued
  fraction then
  \begin{equation*}
    \label{eq:10}
\{\mathcal{M}_+(\alpha,\beta):\beta\in {\mathbf R} \}=[0,\mathcal{M}_+(\alpha)].    
  \end{equation*}
\end{thm}
In this case, we avoid the problems of long sequences of $2$s in the
negative continued fraction by making use of the Davenport expansion
of $\beta$ with respect to $\alpha$ using the ordinary continued
fraction. This theory is described in Rockett and
Sz\"usz~\cite{rockett92:_contin_fract} with a different notation. The
notation we use is largely that of Cassels~\cite{cassels54:ueber} with
$\sharp$ appended to indicate use of the ordinary continued fraction
but with $D^{\sharp}_n$ denoting the quantity he refers to as
$\epsilon_n$. 

Note that, when $\beta=n\alpha+m$ for $n$ and $m$ integers,
$\mathcal M(\alpha,\beta)=0$.  Accordingly, we restrict attention to $\beta$ not of
this form. 

Set $\alpha=[0;a^{\sharp}_1,a^{\sharp}_2,...]$ and let $(n_k)$ be a sequence of indices on which
the partial quotients are strictly monotonically increasing. Now let
$0<c<\mathcal{M}_+ (\alpha)$ and choose $\beta$ for which
$c<\mathcal{M}_+(\alpha,\beta)\le \mathcal{M}_+(\alpha)$. Let its Davenport
coefficients be $(b^{\sharp}_j)$ in the ordinary continued
fraction. We will construct a sequence $(c^{\sharp}_j)$ so that $c^{\sharp}_j=b^{\sharp}_j$ except on a
subsequence of the $n_k$ which will be chosen sufficiently sparse for our
purposes.

Since $\beta=\sum_{k=1}^\infty b^{\sharp}_k D^{\sharp}_k$, where
$D^{\sharp}_k=q^{\sharp}_{k-1}\alpha-p^{\sharp}_{k-1}$, we put
$$\lambda^{\sharp}_n(\beta)=Q^{\sharp}_n\|Q^{\sharp}_n\alpha-\beta\|$$
The ordinary case of (\ref{eq:143}) (see\cite{cassels54:ueber} or
\cite{rockett92:_contin_fract}) gives

\begin{multline}\label{eq:words}
\lambda^{\sharp}_n(\beta)=(\sum_{k=1}^n b^{\sharp}_k q^{\sharp}_{k-1})|\sum_{k={n+1}}^\infty b^{\sharp}_k D^{\sharp}_k|\\
=q^{\sharp}_{n}|D^{\sharp}_n|(b^{\sharp}_n\frac{q^{\sharp}_{n-1}}{q^{\sharp}_n}+b^{\sharp}_{n-1}\frac{q^{\sharp}_{n-2}}{q^{\sharp}_{n-1}}\frac{q^{\sharp}_{n-1}}{q^{\sharp}_n}+...)
|b^{\sharp}_{n+1}\frac{D^{\sharp}_{n+1}}{D^{\sharp}_n}+b^{\sharp}_{n+2}\frac{D^{\sharp}_{n+2}}{D^{\sharp}_n}+...|
\end{multline}

Note that
$q^{\sharp}_n|D^{\sharp}_n|=[a^{\sharp}_n,a^{\sharp}_{n+1},...]/[a^{\sharp}_n,a^{\sharp}_{n-1},a^{\sharp}_{n-2},...,a^{\sharp}_2,a^{\sharp}_1]$,
and so is absolutely bounded above and away from zero. For this choice
of $\beta$, this product is always at least $1/30$ and so the second
two terms in the product are each at least $1/60$. Changing the value
of $b^{\sharp}_n$ by 1 will change the value of $\lambda^{\sharp}_{n}$
by at most $1/(a^{\sharp}_{n}-1)$, so by choosing $n=n_k$ and
adjusting the value of $b^{\sharp}_{n_{k}}$ to $c_{n_{k}}$, we replace
$\beta$ by $\widetilde{\beta}$, so that
\begin{equation*}
  \label{eq:11}
c<\min(\lambda_{n}(\widetilde{\beta}),\lambda_{n-1}(\widetilde{\beta}))<c+2/a^{\sharp}_n.\end{equation*}

By making this change at the indices $n_{k}$ (so that
$a^{\sharp}_{n_{k}}\to\infty$), and putting
$c^{\sharp}_n=b^{\sharp}_n$ elsewhere, we obtain a number
$\gamma=\sum_k c^{\sharp}_k D^{\sharp}_k$ for which
\begin{equation*}
  \label{eq:13}
\mathcal{M}_+(\alpha,\gamma)=c,  
\end{equation*}
since the effect of these changes for other $\lambda_n$ is smaller than
that  at $n=n_k$ or $n=n_{k-1}$. In fact we have:
\begin{lem}\label{lem1} 
  Let $\beta$ have Davenport coefficients $(b^{\sharp}_i)$ in the
  ordinary continued fraction. Given any
  $\epsilon>0$ and $k$ sufficiently large, there is an $M=M(k)<k/2$ and
  $N=N(k)$ such that if $m\not\in (k-M,k+N)$ then any change in $b^{\sharp}_k$ will
  not change $\lambda^{\sharp}_m(\beta)$ or $\rho^{\sharp}_{m}(\beta)$ by more than $\epsilon$.
\end{lem} 

\begin{proof} This follows quickly by (\ref{eq:words}), since $\alpha$ must have
  infinitely many partial quotients in its continued fraction expansion which
  are larger than $2$. If $k$ is sufficiently large then there are at least
  $-\log {\epsilon}/\log 2$ such terms $a^{\sharp}_n$ in $n\in [k/2,k)$ and at least
  $-\log {\epsilon}/\log 2$ such terms $a^{\sharp}_n$ in $(k,k+N]$. Consequently any
  change in $b_k$ will make a variation in the value of $\lambda^{\sharp}_m(\beta)$ and
  $\rho^{\sharp}_{m}(\beta)$ less than $\epsilon$.

   We now choose a sequence of the $n_k$ which are sufficiently sparse that
   these intervals do not overlap. Choose $c_{n_k}$ so that
   \begin{equation*}
     \label{eq:14}
     c<\min_{j\in[n_k-M(n_k),n_k+N(n_k]}(min(\lambda^{\sharp}_{j}(\gamma),{\rho^{\sharp}_j}(\gamma))< c+2/a^{\sharp}_{n_k}.      
   \end{equation*}
   This is clearly possible using the fact that changing $b_k$ by 1 increases
   or decreases the expression in (\ref{eq:words}) by no more than $1/(a^{\sharp}_n-1)$.
   This completes the proof of the fact that for such well approximable
   $\alpha$ the spectrum consists of a single ray.
\end{proof}

\section{The Bounded Case}
\label{sec:6tdas}
In the light of results of the previous section, we restrict attention from this point to
the case where the ordinary continued has bounded partial quotients
($a^\sharp_{n}$).  This translates in the case of the negative continued fraction to
the sequence $a_1,a_2,a_3,\ldots$ being bounded, with least upper bound $M$,
and the lengths of the blocks of consecutive $2$'s also being bounded with
least upper bound $N-1\geq 0$.  Then it follows from equations
\eqref{eq:18} and  \eqref{eq:19} that
\begin{equation}
  \label{eq:26}
  \frac1{M}<\overline\alpha_i<\frac{N}{N+1}
\end{equation}
hold for all $i\ge1$. We choose $L$ to be the smallest integer
such that
\begin{equation}
  \label{eq:175}
  \Bigl(\frac{N}{N+1}\Bigr)^L\le\frac{(1-\frac{N}{N+1})(1-\frac{N^{2}}{(N+1)^{2}})}{M^N(M^2-1)}.
\end{equation}
The numbers $N$ and $L$ will figure significantly in the proof in the
bounded case. 

\subsection{Computation of $\mathcal M_+(\alpha,\beta)$}
\label{sec:scomanbbd}

We will define   a collection of $\beta$'s,  in terms of
their Davenport coefficients, which $\beta$ have  the property that
for some subsequence $(k(i))$ of positive integers
\begin{equation}
  \label{eq:154}
  {\mathcal M}_{+}(\alpha,\beta)=\liminf_{i\to\infty}\lambda_{k(i)}(\beta).
\end{equation}
This enables us to work with just the $\lambda_{k(i)}$ rather than the
$\rho_{n}$ and simplifies the rest of the proof of our main theorem. 
We assume throughout the remainder of the proof of the bounded partial
quotient case that  $\beta\neq n\alpha+m$ for some integers $n$ and $m$.

We record some simple results in the following lemma. 
\begin{lem}
\label{thm:lkadlnla}
  \begin{enumerate}
  \item For $i<j$, 
    \begin{equation}
      \label{eq:15}
q_iD_j
=\frac{\alpha_{i+1}\alpha_{i+2}\ldots\alpha_j}{1-\overline\alpha_i\alpha_{i+1}}.      
    \end{equation}
  \item Let $r$ and $s$ be positive integers satisfying $r\ge sL$.
Then
\begin{equation*}
  \label{eq:176}
  q_uD_{v-1}<q_{n-1}D_m(1-\alpha_{m+1})<q_{n-1}D_m
\end{equation*}
whenever $u$, $v$, $n$ and $m$ are positive integers with
$u+r<v$ and $n<m\le n+s+N$.
  \end{enumerate}
\end{lem}

\begin{proof} The first part is a simple calculation. For the second part,
  note that the right inequality is obviously true since $0<\alpha_{m+1}<1$.
  To prove the left inequality we observe that \eqref{eq:15} implies
\begin{equation*}
  \label{eq:177}
  q_uD_{v-1}=\frac{\alpha_{u+1}\alpha_{u+2}\dots\alpha_{v-1}}{1-\overline\alpha_u\alpha_{u+1}}.
\end{equation*}
Using \eqref{eq:26}, \eqref{eq:24} and  $u+r<v$, we have
\begin{equation*}
  \label{eq:178}
  q_uD_{v-1}<\frac{R^{v-u-1}}{1-R^2}\le\frac{R^r}{1-R^2},
\end{equation*}
where $R=N/(N+1)$.
Similarly, 
\begin{equation*}
  \label{eq:179}
  \begin{aligned}[t]
  q_{n-1}D_m(1-\alpha_{m+1})
 & =\frac{\alpha_{n}\alpha_{n+1}\dots\alpha_m(1-\alpha_{m+1})}
    {1-\overline\alpha_{n-1}\alpha_n}\\
  &>\frac{M^{-(s+N+1)}(1-R)}{1-M^{-2}}.   
  \end{aligned}
\end{equation*}
The lemma is, therefore, true if
\begin{equation*}
  \label{eq:181}
  \frac{R^r}{1-R^2}\le\frac{M^{-(s+N+1)}(1-R)}{1-M^{-2}}
\end{equation*}
Since $r\ge sL$ and $R<1$ and $s\ge1$ and $R^LM<1$ we have
  $R^rM^{s-1}<R^{sL}M^{s-1}<R^L$
and the result follows immediately from the definition of $L$.
\end{proof}

\begin{thm}
\label{thm:98967}
Choose positive integers $r$ and $s$ with $r\ge sL$, and an increasing
sequence of indices $(k(i))$ with $k(i+1)>k(i)+r$. 
Let $0<\beta<1$ with Davenport coefficients
$(b_i)$  satisfy:
\begin{enumerate}
\item for each $i\ge1$ the sequence $b_{k(i)+1},b_{k(i)+2},\ldots,b_{k(i)+r}$
is a block of $r$ zeros;
\item there is no block of $N+s$ consecutive zeros in  $(b_{n})$ between
  $k(i)+r$ and $k(i+1)$,
\item $\beta$ is not in short intervals at level $n$ for $N+s$ consecutive
  values of $n$, in other words the Davenport coefficients of $\beta$ contain no
  sequence of the form $a_{j}-1,a_{j+1}-2,\ldots,a_{j+N+s-1}-2$.
\end{enumerate}
then
\begin{equation*}
  \label{eq:184}
  {\mathcal M}_{+}(\alpha,\beta)=\liminf_{i\to\infty}\lambda_{k(i)}(\beta).
\end{equation*}
\end{thm}

\begin{proof}
By Theorem~\ref{thm:333}, it is enough to show that 
\begin{equation*}
  \label{eq:186}
  \lambda_n(\beta)\ge\lambda_{k(i)}(\beta)\qquad\text{or}\qquad
  \lambda_n(\beta)\ge\lambda_{k(i+1)}(\beta)
\end{equation*}
and
\begin{equation}
  \label{eq:187}
  \rho_n(\beta)\ge\lambda_{k(i)}(\beta)
\end{equation}
for all integers $n$ with $k(i)\le n<k(i+1)$, for $i$ sufficiently large.  We
choose $i>i_{0}$ to ensure that some $b_{j}\neq 0$ for some $j<i_{0}$ and that
$\beta$ has appeared in a long interval before that stage. If this were not
possible $\beta$ would be a multiple of $\alpha$ modulo 1.  Now fix $n$
between $k(i)$ and $k(i+1)$. We will liberally use the fact stated in
Lemma~\ref{thm:6795} that we can move back and forth between
$\lambda_{n}(\beta)$ and $\lambda_{m}(\beta)$ provided the intervening $b_{k}$
are all zero. Similarly, at the other extreme, we could move back and forward
between $\rho_{n}(\beta)$ and $\rho_{m}(\beta)$ provided that at the
intervening levels $\beta$ is in short intervals.

Choose $u\leq k(i)< v$ to be such that $b_{j}=0$ if $u<j<v$ and to be the
extreme integers with that property.  We observe that $v-u>r$. It follows from
Lemma~\ref{thm:6795} that 
\begin{equation*}
  \label{eq:12}
  \lambda_{k(i)}<q_{u}D_{v-1}. 
\end{equation*}
If $n<k(i)+r$ then $\lambda_{k(i)}=\lambda_{n}$. If not, then $b_{n}$ is
followed by a block of at most $N+s$ zeros unless $b_{m}=0$ for all $m$ with
$n<m<k(i+1)$, in which case $\lambda_{k(i+1)}=\lambda_{n}$.  If
$\lambda_{k(i)}\neq\lambda_{n}\neq \lambda_{k(i+1)}$ then
\begin{equation*}
  \label{eq:190}
  q_{n-1}D_m\le\lambda_n(\beta),
\end{equation*}
for some $m\le n+N+s$. 
That
$\lambda_{k(i)}(\beta)\le\lambda_n(\beta)$ follows from 
\begin{equation*}
  \label{eq:191}
  q_uD_{v-1}\le q_{n-1}D_m.
\end{equation*}
which  follows immediately from  $m\le n+s+N$ and  $u+r<v$.

The argument to show that 
\eqref{eq:187} holds when $k(i)\le n<k(i+1)$ is similar but uses the fact that 
$\beta$ is not in  a long sequence of consecutive short intervals. 
\end{proof}

\subsection{Elements of ${\mathcal S}_{+}(\alpha)$}

Now we give a construction for  certain elements of ${\mathcal S}_{+}(\alpha)$ using
Theorem~\ref{thm:98967}.  First we impose additional constraints on the
sequence $(k(i))$ so that the limits of the sequences $\overline\alpha_{k(i)}$
\eqref{eq:24} and $\alpha_{k(i)+1}$ both exist.  Moreover, the limits lie
strictly between $0$ and $1$, since \eqref{eq:26} and \eqref{eq:24} hold for
all $i\ge1$ and $0<1/M<N/(N+1)<1$.
The collection of  $\beta$ to be described in terms of their Davenport
expansions will  be the  ones for which $\mathcal M(\alpha,\beta)$ are in the Hall's Ray.

\begin{defn}
We choose $(K(i))$ be an increasing sequence of indices with gaps
$K(i+1)-K(i)$ tending to infinity
such that the limits 
\begin{equation}
  \label{eq:198}
  \begin{aligned}
    a_1^-,a_2^-,a_3^-,\ldots&
    =\lim_{i\to\infty}a_{K(i)},a_{K(i)-1},\ldots,a_2,a_1\ldots\\
    a_1^+,a_2^+,a_3^+,\ldots
    &=\lim_{i\to\infty}a_{K(i)+1},a_{K(i)+2},a_{K(i)+3},\ldots,
  \end{aligned}
\end{equation}
exist; that is, that in each case the sequence of integers eventually
becomes constant. The existence of such a sequence follows quickly by
a diagonal argument from the finiteness of the alphabet from which the
$a_i$'s are chosen. 
\end{defn}
We write
\begin{equation}
  \label{eq:197}
  \alpha^-=\langle a_1^-,a_2^-,a_3^-,\ldots\rangle \qquad\text{and}\qquad
  \alpha^+=\langle a_1^+,a_2^+,a_3^+,\ldots\rangle.
\end{equation}
The following lemma is a straightforward consequence of the properties
of the sequence $a_1,a_2,a_3,\ldots$

\begin{lem}
  \label{thm:616161}
Each of the sequences $\alpha^-$ and $\alpha^+$ have all of their
terms less than or equal to $M$ and contain no block of $N$
consecutive $2$'s.   
\end{lem}
Evidently, the numbers $\alpha^-$ and $\alpha^+$ are irrational with 
$0<\ \alpha^{\pm}<1$, and the  partial quotients of
their regular continued fraction expansions satisfy \eqref{eq:26}.

All of the theory in the preceding sections is applicable to $\alpha^-$ or
$\alpha^+$ in place of $\alpha$.  We introduce the following notation.  For
$i\ge1$, define
\begin{equation}
  \label{eq:200}
  \alpha^-_i=\langle a^-_i,a^-_{i+1},a^-_{i+2},\ldots\rangle\qquad\text{and}\qquad
  \alpha^+_i=\langle a^+_i,a^+_{i+1},a^+_{i+2},\ldots\rangle
\end{equation}
and  set
\begin{equation}
  \label{eq:201}
  D^-_i=\alpha^-_1\alpha^-_2\ldots\alpha^-_i\qquad\text{and}\qquad
  D^+_i=\alpha^+_1\alpha^+_2\ldots\alpha^+_i.
\end{equation}
It follows from \eqref{eq:197}, \eqref{eq:200}, and
the discussion above that
\begin{equation*}
  \label{eq:203}
  \alpha^-_k=\lim_{i\to\infty}\overline\alpha_{K(i)-k+1}\qquad\text{and}\qquad
  \alpha^+_k=\lim_{i\to\infty}\alpha_{K(i)+k}.
\end{equation*}
Hence
\begin{equation*}
  \label{eq:204}
  \begin{aligned}
   D^-_k&=\lim_{i\to\infty}\overline\alpha_{K(i)}\overline\alpha_{K(i)-1}\ldots\overline\alpha_{K(i)-k+1}
   =\lim_{i\to\infty}\frac{q_{K(i)-k}}{q_{K(i)}}\\
  D^+_k&=\lim_{i\to\infty}\alpha_{K(i)+1}\alpha_{K(i)+2}\ldots\alpha_{K(i)+k}
  =\lim_{i\to\infty}\frac{D_{K(i)+k}}{D_{K(i)}}.
  \end{aligned}
\end{equation*}
The next lemma,  a crucial
one in the proof, makes use of these identities.  

\begin{lem}
\label{thm:67023}
Let $(b_i)$ be the Davenport coefficients of a number
$\beta\in[0,1]$ for which  both of the limits
\begin{equation*}
  \begin{aligned}[t]
  b^-_1,b^-_2,b^-_3,\ldots
  &=\lim_{i\to\infty}b_{K(i)},b_{K(i)-1},\ldots,b_1,0,0,0,\ldots\\
  b^+_1,b^+_2,b^+_3,\ldots
  &=\lim_{i\to\infty}b_{K(i)+1},b_{K(i)+2},b_{K(i)+3},\ldots
  \end{aligned}
  \label{eq:206}
\end{equation*}
exist and let
\begin{equation*}
  \label{eq:208}
  \beta^-=\sum^\infty_{k=1}b^-_kD^-_k\qquad\text{and}\qquad
  \beta^+=\sum^\infty_{k=1}b^+_kD^+_k.
\end{equation*}
Then
\begin{equation*}
  \label{eq:209}
  \lim_{i\to\infty}\lambda_{K(i)}(\beta)=\frac{\beta^-\beta^+}{1-\alpha^-\alpha^+}.
\end{equation*}
\end{lem}

\begin{proof}
By definition
\begin{equation*}
  \label{eq:210}
  \lambda_{K(i)}(\beta)=Q_{K(i)}D_{K(i)}\beta_{K(i)+1}
  =\frac{Q_{K(i)}}{q_{K(i)}}q_{K(i)}D_{K(i)}\beta_{K(i)+1}
\end{equation*}
and, by  \eqref{eq:200},
\begin{equation*}
  \label{eq:212}
  \lim_{i\to\infty}q_{K(i)}D_{K(i)}=\frac1{1-\alpha^-\alpha^+}.
\end{equation*}
In consequence,  it is sufficient to observe that 
\begin{equation*}
  \label{eq:213}
  \lim_{i\to\infty}\frac{Q_{K(i)}}{q_{K(i)}}=\beta^-\qquad\text{and}\qquad
  \lim_{i\to\infty}\beta_{K(i)+1}=\beta^+.
\end{equation*}
This is a straightforward consequence of the fact that 
\begin{equation*}
  \label{eq:214}
  D_{K(i)}\beta_{K(i)+1}=\sum^\infty_{k=1}b_{K(i)+k}D_{K(i)+k}
\end{equation*}
and a corresponding expression for the first limit. 
\end{proof}

Now we define two Cantor-like subsets of $[0,1)$ in terms of their Davenport
expansions.
\begin{defn}
\begin{enumerate}
\item
$\beta\in E(\alpha,s)$ if and only if
in its Davenport coefficents $(b_i)$
no block $b_i,b_{i+1},\dots,b_{i+s}$ consists solely of zeros
or is of the form
\begin{equation}
  \label{eq:225}
  a_i-2,a_{i+1}-2,\dots,a_{i+s-1}-2,a_{i+s}-1.
\end{equation} Note that this does not preclude tail sequences of the form $a_{i}-1,
a_{i+1}-2, a_{i+2}-2, \ldots$. 
\item
$\beta\in F(\alpha,s)$ if and only if
in the sequence $b_1,b_2,b_3,\dots$
no block $b_i,b_{i+1},\dots,b_{i+s}$ consists solely of zeros
or is of the form
\begin{equation*}
  \label{eq:226}
  a_i-1,a_{i+1}-2,a_{i+2}-2,\dots,a_{i+s}-2.
\end{equation*}
\end{enumerate}
We note that both of $F(\alpha,s)$ and $E(\alpha,s)$  are closed
subsets of $[0,1]$. 
\end{defn}

We  now state and prove the main result of this section.
\begin{thm}
\label{thm:main_r_s}
Let $r$ and $s$ be positive integers which satisfy $s\ge N$ and
$r\ge sL$ and let $\alpha^-$ and $\alpha^+$ be defined by \eqref{eq:200} and
$\alpha^+_r$ by \eqref{eq:200}  and $D^+_r$ by \eqref{eq:201}.
For all $e\in E(\alpha^-,s)$ and $f\in F(\alpha^+_{r+1},s)$
there is some $\beta$ with $0<\beta<1$ such that
\begin{equation}
  \label{eq:241}
  {\mathcal M}_{+}(\alpha,\beta)=\frac{efD^+_r}{1-\alpha^-\alpha^+}.
\end{equation}
\end{thm}

\begin{proof}
We will exhibit appropriate Davenport expansions of $\beta$ to achieve this 
result for $f\in F(\alpha^+_{r+1},s)$ and $e\in  E(\alpha^-,s)$. 

Let $e\in E(\alpha^-,s)$ and $f\in F(\alpha^+_{r+1},s)$.
We shall prove there is a $\beta$ with $0<\beta<1$ which
satisfies \eqref{eq:241} by
constructing its Davenport coefficients $(b_i)$.
Specifically, we shall construct $b_1,b_2,b_3,\ldots$ so that
the limits
\begin{equation*}
  \label{eq:242}
  \begin{aligned}[t]
   b^-_1,b^-_2,b^-_3,\ldots
  &=\lim_{i\to\infty}b_{K(i)},b_{K(i)-1},\ldots,b_1,0,0,0,\ldots\\
   b^+_1,b^+_2,b^+_3,\ldots
  &=\lim_{i\to\infty}b_{K(i)+1},b_{K(i)+2},b_{K(i)+3},\ldots
  \end{aligned}
\end{equation*}
exist and
\begin{equation}
  \label{eq:244}
  e=\sum^\infty_{k=1}b^-_kD^-_k\qquad\text{and}\qquad
  fD^+_r=\sum^\infty_{k=1}b^+_kD^+_k.
\end{equation}
Lemma~\ref{thm:67023} then yields:
\begin{equation}
  \label{eq:245}
  \lim_{i\to\infty}\lambda_{K(i)}(\beta)=\frac{efD^+_r}{1-\alpha^-\alpha^+}.
\end{equation}

We  describe sequences
$b^+_1,b^+_2,b^+_3,\ldots$ and $b^-_1,b^-_2,b^-_3,\ldots$
for which \eqref{eq:245} holds.
Let $f_1,f_2,f_3,\ldots$ be the Davenport coefficients of $f$ with respect
to $\alpha^+_{r+1}$
and  observe that
\begin{equation*}
  \label{eq:247}
  f=\sum^\infty_{k=1}f_k\alpha^+_{r+1}\alpha^+_{r+2}\ldots\alpha^+_{r+k}.
\end{equation*}
Multiplication  by $D^+_r$  gives
\begin{equation*}
  \label{eq:248}
  fD^+_r=\sum^\infty_{k=1}f_kD^+_{r+k}.
\end{equation*}
and therefore the right hand formula in \eqref{eq:244} holds if we define
\begin{equation}
  \label{eq:249}
  b^+_1,b^+_2,b^+_3,\ldots=\underbrace{0,\ldots,0}_{r},f_1,f_2,f_3,\ldots.
\end{equation}
It is easily seen that these satisfy the appropriate conditions for a Davenport
expansion. 

For a number $e\in E(\alpha^-,s)$, 
we let $e_1,e_2,e_3,\ldots$ be the $\alpha^-$-expansion of $e$
and as above we observe that
\begin{equation*}
  \label{eq:252}
  e=\sum^\infty_{k=1}e_kD^-_k.
\end{equation*}
The left hand formula in \eqref{eq:244} then holds if we set
\begin{equation*}
  \label{eq:253}
  b^-_1,b^-_2,b^-_3,\ldots=e_1,e_2,e_3,\ldots.
\end{equation*}

Next, we specify enough of the sequence $b_1,b_2,b_3,\ldots$ to ensure that
\eqref{eq:249} and \eqref{eq:244} hold. At this point Figure~\ref{fig:ls}
illustrates definition of  the various pieces of the sequence.
\begin{figure}[ht!]
  \centering
\includegraphics[width=0.9\textwidth]{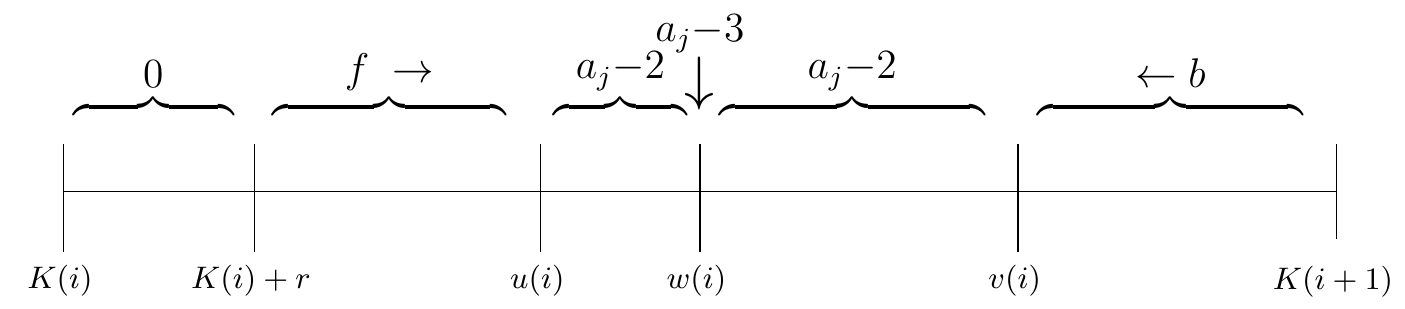}
\caption{The definition of the sequence $b_{i}$.}
  \label{fig:ls}
\end{figure}

For this purpose, we choose a positive integer $i_0$ and sequences of integers
$(u(i))^\infty_{i=i_0}$ and $(v(i))^\infty_{i=i_0}$ such that
$  K(i)\le u(i)<u(i)+N<v(i)\le K(i+1)$
for all $i\ge i_0$ and
\begin{equation*}
  \label{eq:257}
  \lim_{i\to\infty}u(i)-K(i)=\infty\qquad\text{and}\qquad
  \lim_{i\to\infty}K(i+1)-v(i)=\infty.
\end{equation*}
Such sequences exist since the differences $K(i+1)-K(i)$
tend to infinity as $i$ increases.
Furthermore we can
also assume that, for all $i\ge i_0$,
\begin{equation*}
  \label{eq:258}
  \begin{aligned}[t]
   a_{K(i)+1},a_{K(i)+2},\ldots,a_{u(i)}
  &=a^+_1,a^+_2,\ldots,a^+_{u(i)-K(i)}\\
  a_{K(i+1)},a_{K(i+1)-1},\ldots,a_{v(i)}
  &=a^-_1,a^-_2,\ldots,a^-_{K(i+1)-v(i)+1}.
  \end{aligned}
\end{equation*}
We ensure \eqref{eq:249} and \eqref{eq:244} hold by defining
\begin{equation*}
  \label{eq:260}
  \begin{aligned}[t]
  b_{K(i)+1},b_{K(i)+2},\ldots,b_{u(i)}
  &=b^+_1,b^+_2,\ldots,b^+_{u(i)-K(i)}\\
  b_{K(i+1)},b_{K(i+1)-1},\ldots,b_{v(i)}
  &=b^-_1,b^-_2,\ldots,b^-_{K(i+1)-v(i)+1}
  \end{aligned}
\end{equation*}
for all $i\ge i_0$.

Before completing the specification of $(b_j)$ we
further restrict $i_0$ and
the sequences $(u(i))^\infty_{i=i_0}$ and $(v(i))^\infty_{i=i_0}$.
\begin{equation*}
  \label{eq:262}
  b_{u(i)}\ne0\qquad\text{and}\qquad b_{v(i)}\ne0
\end{equation*}
for all $i\ge i_0$. This is  relatively easy to arrange from the
properties  of the $K(i)$  in relation to the Davenport expansion,  and of the
sequences $(u(i))$ and $(v(i))$. To complete the specification of $(b_j)$
we introduce one more sequence.
We choose the sequence $(w(i))^\infty_{i=i_0}$ so that
\begin{equation*}
  \label{eq:265}
  u(i)<w(i)\le u(i)+N \text{ and }   a_{w(i)}\ge3 \quad (i\ge i_0).
\end{equation*}
Such a choice is clearly possible.

We can now unambiguously define
\begin{equation*}
  \label{eq:267}
  b_j=
  \begin{cases}
0&\text{ if $1\le j\le K(i_0)$}\\
a_j-2&\text{ if $u(i)<j<w(i)$ for some $i\ge i_0$}\\
a_j-3&\text{ if $j=w(i)$ for some $i\ge i_0$}\\
a_j-2&\text{if $w(i)<j<v(i)$ for some $i\ge i_0$.}
  \end{cases}
\end{equation*}

It is not hard to verify that $0\le b_i<a_i$ for all $i\ge1$ and
since $b_{w(i)}=a_{w(i)}-3$ for all $i\ge i_0$
it is also clear that no subsequence $b_i,b_{i+1},b_{i+2},\ldots$
is of the form \eqref{eq:70}. It is easy to check that 
$(b_i)$ are  Davenport coefficients
by showing that no block $b_i,b_{i+1},\ldots,b_j$
is of the form \eqref{eq:69}. 

Now we observe that 
the hypotheses of Theorem~\ref{thm:98967} holds with $k(i)=K(i)$ for
all $i$ 
to complete the proof. 
\end{proof}

\subsection{Cantor dissections for $ E(\alpha,s)$ and $F(\alpha,s)$}

Our  eventual aim is to  show that if the integer $s$ is large enough then the
product of the two sets $ E(\alpha^-,s)$ and $F(\alpha^+_{r+1},s)$, where
$r\ge1$, contains an interval.  Towards that aim we describe each of these two
sets in terms of Cantor dissections. We do this for a generic $\alpha$ rather
than $\alpha^{-}$ and $\alpha^{+}$ at this stage.  We collect together a few
definitions.
\begin{defn}
\label{defn:sss}
  \begin{enumerate}
   \item $H(\alpha,s)$ and $G(\alpha,s)$ are the smallest closed intervals
    containing $F(\alpha,s)$ and $E(\alpha,s)$ respectively.
  \item For each sequence $\bc_n=c_1,c_2,\ldots,c_n$ of positive integers,define:
    \begin{equation*}
      \label{eq:288}
      \begin{aligned}[t]
        S(\bc_n)&=\sum^n_{k=1}c_kD_k,\\
        F(\bc_n)&=\{\gamma=\sum_{k=1}^{\infty}b_{k}D_{k}\in 
        F(\alpha,s):b_{k}=c_{k},\ (k=1,2, \ldots,n)\}
      \end{aligned}
    \end{equation*}
    where $\sum_{k=1}^{\infty}b_{k}D_{k}$ is the Davenport expansion of $\gamma$.  Denote by 
    $C(\bc_n)$ the smallest closed interval which contains
    $F(\bc_n)$. Observe that $C(\bc_n)$ may be the empty set.
  \item when $(\bc_n)\neq \emptyset$, $$C(\bc_n)= [\underline C(\bc_n),  \overline C(\bc_n)]$$ where $$\underline C(\bc_n)=\inf C(\bc_n),\ \overline C=\sup C(\bc_n)\text{ and }  |C(\bc_n|=\overline C(\bc_n)-\underline C(\bc_n).$$
  \end{enumerate}
We allow the possibility that $n=0$ in which case $C(\;)=H(\alpha,s)$.
\end{defn}

The dissection of $H(\alpha,s)$ to obtain $F(\alpha,s)$ begins
by replacing $C(\;)=H(\alpha,s)$ with the collection of intervals
\begin{equation*}
  \label{eq:291}
  \{C(0),C(1),\ldots,C(a_1-1)\}.
\end{equation*}
The $n$-th stage of the dissection 
replaces each non-empty interval $C(\bc_n)$
by the collection of intervals
\begin{equation}
  \label{eq:292}
  \{C(\bc_{n+1}):\;0\le c_{n+1}<a_{n+1}\}.
\end{equation}
From the definition of $C(\bc_n)$ it is clear that it is the
smallest closed interval containing the collection of intervals
\eqref{eq:292}. Moreover the restrictions on the digits results in gaps
between all of these. As an illustration, note that if
$C(\bc_{n})$ has the last $n-1$ $c_{k}$ all equal to $0$, then at the next
level $C(\bc_{n},0)=\emptyset$. The same kind of phenomenon
occurs at the opposite end because of the restriction on the number of
terms of the form
$a_{i}-2$. It is clear that this is a Cantor dissection that produces
$F(\alpha,s)$, and we have 
\begin{equation*}
  \label{eq:294}
  \begin{gathered}[t]
    S(\bc_n)+D_{n+s+1}\le\underline C(\bc_n)\leq \overline C(\bc_n)\le S(\bc_n)+D_n\\
    C(\bc_{n+1})\subset
    [S(\bc_n)+c_{n+1}D_{n+1}+D_{n+s+2},\;S(\bc_n)+(c_{n+1}+1)D_{n+1}].
  \end{gathered}
\end{equation*}
Clearly $|C(\bc_n)|\le D_n$, and  it is evident that
\begin{equation*}
  \label{eq:296}
  F(\alpha,s)=\bigcap^\infty_{n=1}\bigcup
  \{C(\bc_n)\ne\emptyset:\;0\le c_i<a_i\}.
\end{equation*}

Now we obtain more precise estimates of the values of the endpoints
$\underline C(\bc_n)$ and $\overline C(\bc_n)$. 
\begin{lem}
\label{thm:ousds}
  Let $s\ge N$, $C(\bc_n)\ne\emptyset$, $t$ the largest integer with
  $0\le t\le n$ such that all of $c_{n-t+1},c_{n-t+2},\ldots,c_n$ are zero and
  $u$  the unique integer with $0\le u\le n$ such that
  $c_{n-u+1},c_{n-u+2},\ldots,c_n$ is equal to
\begin{equation*}
  \label{eq:331}
  a_{n-u+1}-1,a_{n-u+2}-2,a_{n-u+3}-2,\ldots,a_n-2.
\end{equation*}
Then
\begin{equation*}
  \label{eq:329}
  \underline C(\bc_n)<\left\{\alignedat2
  &S(\bc_n)+D_{n+s}\quad&&\text{if $t=0$}\\
  &S(\bc_n)+D_{n+1}+D_{n+s}\quad&&\text{if $t>0$}
  \endalignedat\right.
\end{equation*}
and
\begin{equation*}
  \label{eq:330}
  \overline C(\bc_n)>\left\{\alignedat2
  &S(\bc_n)+D_n-D_{n+s-N}\quad&&\text{if $u=0$}\\
  &S(\bc_n)+D_{n+N+1}\quad&&\text{if $u>0$}
  \endalignedat\right.
\end{equation*}
\end{lem}

\begin{proof} Write $C=C(\bc_n)$ and note that 
$\underline C=\underline C(\bc_n)$ is the number $\beta$
whose Davenport coefficients $(b_i)$ are of the form
\begin{equation}
  \label{eq:332}
  c_1,c_2,\ldots,c_n,\underbrace{0,\ldots,0}_{s-t},1,\underbrace{0,\ldots,0}_{s},1,
  \underbrace{0,\ldots,0}_s,1,\ldots.
\end{equation}
Note that $t\le s$ else $c_1,c_2,\ldots,c_n$ ends with
more than $s$ consecutive zeros and $C=\emptyset$. In other words, 
\begin{equation*}
  \label{eq:333}
  \underline C=\sum^n_{k=1}c_kD_k+D_{n+s-t+1}+D_{n+2s-t+2}+D_{n+3s-t+3}+\ldots,
\end{equation*}
and
\begin{equation*}
  \label{eq:334}
  \underline C\le\left\{\alignedat2
  &S(\bc_n)+D_{n+s+1}+D_{n+2s+2}+D_{n+3s+3}+\ldots
    \quad&&\text{if $t=0$}\\
  &S(\bc_n)+D_{n+1}+D_{n+s+2}+D_{n+2s+3}+\ldots
    \quad&&\text{if $t>0$}\endalignedat\right.
\end{equation*}
Further, since $s\ge N$ we know
\begin{equation*}
  \label{eq:335}
  D_{n+s}>D_{n+s+1}+D_{n+2s+2}+D_{n+3s+3}+\ldots,
\end{equation*}
and 
\begin{equation*}
  \label{eq:336}
  D_{n+s}>D_{n+s+2}+D_{n+2s+3}+D_{n+3s+4}+\ldots, 
\end{equation*}
and the truth of the first statement of the lemma is evident.

We describe the Davenport expansion of $\overline C=\overline C(\bc_n)$ next. 
Let $k(0)=n-u$ and  inductively define the sequence
$k(1),k(2),k(3),\ldots$ by choosing $k(i)$ to be the largest integer such that
\begin{equation*}
  \label{eq:337}
  k(i-1)+2\le k(i)\le k(i-1)+s+1\qquad\text{and}\qquad a_{k(i)}\ge3.
\end{equation*}
This is possible by the properties of $a_n$ as enunciated in
Lemma~\ref{thm:616161} for $\alpha=\alpha^-$ and $\alpha=\alpha^+$.  Further,
\begin{equation*}
  \label{eq:338}
  k(i)\ge k(i-1)+s-N+2
\end{equation*}
because if not $a_k=2$ for all $k$ between and including
$k(i-1)+s-N+2$ and $k(i-1)+s+1$ contrary to the definition of $N$.
Now  $\overline C$ is the number $\beta$ whose
Davenport coefficients $(b_i)$ are  defined by
\begin{equation*}
  \label{eq:339}
  b_i=
  \begin{cases}
c_i&\text{ if $i\le k(0)$}\\
a_i-1&\text{ if $i=k(j)+1$ for some $j\ge0$}\\
a_i-2&\text{ if $k(j)+1<i<k(j+1)$ for some $j\ge0$}\\
a_i-3&\text{ if $i=k(j)$ for some $j\ge1$.}
  \end{cases}
\end{equation*}
These are clearly Davenport coefficients, and the sequence contains no block
$b_i,b_{i+1},\ldots,b_{i+s}$ of the form \eqref{eq:225} nor does it contain a
block of $b_i,b_{i+1},\ldots,b_{i+s}$ consisting solely of zeros. We conclude
that $\beta\in F(\alpha,s)$. It is also fairly clear that the sequence
$b_1,b_2,b_3,\ldots$ begins with $c_1,c_2,\ldots,c_n$.

It remains to show that no other element of $C(\bc_n)$ is larger
than $\beta$.  If that were the case, and there were some $\beta'\in
C(\bc_n)$ with Davenport coefficients $(b'_i)$ such
that $\beta'>\beta$. However, the form of the definition of $\beta$ prohibits any
possible increase in the values of the Davenport coefficients  while staying a
member of $F(\alpha,s)$ and starting with $c_{1}, c_{2},\ldots, c_{n}$. 
Evidently $\underline C=\sum^\infty_{k=1}b_kD_k$.
By truncating this series at the term with index $k(1)+1$
and making some minor rearrangements we find that
\begin{equation*}
  \label{eq:344}
  \overline C>\sum^{k(0)}_{l=1}c_lD_l+\sum^{k(1)+1}_{l=k(0)+1}(a_l-2)D_l
  +D_{k(0)+1}-D_{k(1)}+D_{k(1)+1}.
\end{equation*}
We consider two cases.
First we suppose $u=0$ and hence $k(0)=n$.
In this case, since
\begin{equation*}
  \label{eq:345}
  D_n<\sum^{k(1)+1}_{l=n+1}(a_l-2)D_l+D_{n+1}+D_{k(1)+1}
\end{equation*}
we obtain 
$\overline C>S(\bc_n)+D_n-D_{k(1)}$.
It is easy to deduce from \eqref{eq:332} with $i=1$ that
$D_{k(1)}>D_{n+s-N}$ and the second statement of the lemma is proved.
Now we suppose $u>0$ and hence $k(0)<n$. Since $k(1)\ge n+1$, 
\begin{equation*}
  \label{eq:346}
  \overline C>\sum^n_{l=1}c_lD_l
  +\sum^{k(1)+1}_{l=n+1}(a_l-2)D_l-D_{k(1)}+D_{k(1)+1}.
\end{equation*}
As $a_{k(1)}\ge3$ and $a_{k(1)+1}\ge2$ we have
\begin{equation*}
  \label{eq:347}
  \overline C>\sum^n_{l=1}c_lD_l+\sum^{k(1)-1}_{l=n+1}(a_l-2)D_l+D_{k(1)+1}.
\end{equation*}
The definition of $N$ implies there is some $i$
with $n+1\le i\le n+N+1$ such that $a_i\ge3$ and so
\begin{equation*}
  \label{eq:348}
  \sum^{n+N+1}_{l=n+1}(a_l-2)D_l\ge D_{n+N+1}.
\end{equation*}
It follows that if $k(1)-1\ge n+N+1$ then
the second statement of the lemma is true.
If on the other hand $k(1)<n+N+1$ then $D_{k(1)+1}\ge D_{n+N+1}$ and
again the truth of the second statement is clear.
This completes the proof the lemma.
\end{proof}

A key remark about the $F(\alpha,s)$ construction, that will not be
true for the case if $E(\alpha,s)$, is that, at least generically, the
deleted intervals resulting from the ``zeros'' condition and the
``$a_n-2$'' condition in this Cantor construction have the property
that their left and right endpoints respectively are
$S(c_1,...,c_n+1)$.

Next we deal with the set $E(\alpha,s)$. This is a little more
complicated; the two restrictions on the
Davenport coefficients of the elements of $E(\alpha,s)$ no longer correspond to
a single gap in the dissection of $G(\alpha,s)$.    We make the following
definition.
\begin{defn}  $A(\;)=[0, 1-D_{1}]$ and $B(\;)=[1-D_{1},1]$.
For each sequence $\bc_n=c_1,c_2,\ldots,c_n$ of positive integers, we
define  
 $A(\bc_n)$ to be the smallest closed interval containing
 $E(\alpha,s)\cap [S(\bc_n),S(\bc_n)+D_n-D_{n+1}]$, and $B(\bc_n)$  to
 be the smallest closed interval containing $E(\alpha,s)\cap [S(\bc_n)+D_{n}-D_{n+1},S(\bc_n)+D_n]$.
\end{defn}

The dissection of $G(\alpha,s)$ begins by replacing $G(\alpha,s)$
with the pair of intervals $A(\;)$ and $B(\;)$.
The next step is the substitution
\begin{equation*}
  \label{eq:308}
  \begin{aligned}[t]
  A(\;)&\to  \{A(0),B(0),A(1),B(1),\ldots,A(a_1-3),B(a_1-3),A(a_1-2)\}   \\
  B(\;)&\to  \{B(a_1-2),\;A(a_1-1)\}.
  \end{aligned}
\end{equation*}

The $n$-th step of the dissection is
\begin{equation}
  \label{eq:310}
  \begin{aligned}
\emptyset\neq A(\bc_{n})&\to \{A(\bc_{n+1}):\;0\le c_{n+1}\le a_{n+1}-2\}\\
&\qquad \cup\{B(\bc_{n+1}):\;0\le c_{n+1}\le a_{n+1}-3\}\\
\emptyset\neq A(\bc_{n})&\to  \{B(c_1,c_2,\ldots,c_n,a_{n+1}-2),\;
  A(c_1,c_2,\ldots,c_n,a_{n+1}-1)\}.
\end{aligned}
\end{equation}
where again we use the notation $\bc_n=c_1,c_2,\ldots,c_n)$.

We note that
$A(\bc_n)$ and $B(\bc_n)$
are the smallest closed intervals containing
the collections  \eqref{eq:310} at the previous level.
This follows because
\begin{equation*}
  \label{eq:312}
  S(\bc_n)+D_n-D_{n+1}
  =S(c_1,c_2,\ldots,c_n,a_{n+1}-2)+D_{n+1}-D_{n+2}.
\end{equation*}
For the moment, we write
\begin{equation*}
  \label{eq:315}
  A=A(\bc_n)\qquad B=B(\bc_n)\qquad
 S=S(\bc_n).
\end{equation*}
Now let $\beta\in E(\alpha,s)$ and suppose its sequence of Davenport coefficients
$(b_i)$ begins with $c_1,c_2,\ldots,c_n$.
Then 
the block $b_{n+1},b_{n+2},\ldots,b_{n+s+1}$ does not
consist entirely of zeros, and so $b_i\ge1$ and $\beta\ge S+D_i$
for some $i$ with $n+1\le i\le n+s+1$.
Hence $\beta\ge S+D_{n+s+1}$.
Since $A$ is the smallest closed interval containing all such numbers $\beta$
which also satisfy $\beta\le S+D_n-D_{n+1}$
it follows that
\begin{equation}
  \label{eq:316}
  S+D_{n+s+1}\le\underline A\leq \overline A\le S+D_n-D_{n+1}.
\end{equation}
In particular $|A|\le D_n$. 
If, on the other hand,  $\beta>S+D_n-D_{n+1}$, there is $i\ge n+1$ such that
the block $b_{n+1},b_{n+2},\ldots,b_i$ is of the form
\begin{equation*}
  \label{eq:317}
  a_{n+1}-2,a_{n+2}-2,\ldots,a_{i-1}-2,a_i-1.
\end{equation*}
We  conclude that
\begin{equation}
  \label{eq:320}
  S+D_n-D_{n+1}+D_{n+s+1}\le\underline B\leq
  \overline B\le S+D_n.
\end{equation}

In fact, 
  \begin{align*}
  \label{eq:321}
    A(\bc_{n+1})&\subset   [S+c_{n+1}D_{n+1}+D_{n+s+2},\;S+(c_{n+1}+1)D_{n+1}-D_{n+2}]\\
    B(\bc_{n+1})&\subset
    [S+(c_{n+1}+1)D_{n+1}-D_{n+2}+D_{n+s+2},\;S+(c_{n+1}+1)D_{n+1}].
  \end{align*}

We note that all such intervals where $0\le c_{n+1}<a_{n+1}$ are disjoint, and
since
\begin{equation}
  \label{eq:323}
   E(\alpha,s)=\bigcap^\infty_{n=1}\bigcup\{
  A(\bc_n)\ne\emptyset,\;B(\bc_n)\ne\emptyset
  :\;0\le c_i<a_i\},
\end{equation}
it is totally disconnected. Again the gaps arise because of the constraints on
digits in the definition of $E(\alpha,s)$.

Now we need to find estimates for the endpoints of the intervals $A(\bc_n)$
and $B(\bc_n)$, just as we have for $C(\bc_n)$ in
Lemma~\ref{thm:ousds}. 
\begin{lem}
\label{thm:yahdd}
Let $s\ge N$ and $A(\bc_n)\ne\emptyset$.
Then
\begin{equation*}
  \label{eq:369}
  \underline A(\bc_n)<S(\bc_n)+D_n-D_{n+1}-D_{n+3N}
\end{equation*}
and
\begin{equation*}
  \label{eq:370}
  \overline A(\bc_n)=S(\bc_n)+D_n-D_{n+1}.
\end{equation*}
Further,
\begin{equation*}
  \label{eq:371}
  \underline A(\bc_n)<S(\bc_n)+D_{n+s}
\end{equation*}
whenever $n=0$ or $B(c_1,c_2,\ldots,c_{n-1},c_n-1)\neq \emptyset$. 
\end{lem}

\begin{proof}
Write $A=A(\bc_n)$.
We note first that  $\underline A$ is the number $\beta$ whose Davenport coefficients
$(b_i)$ are of the form
\begin{equation*}
  \label{eq:372}
  c_1,c_2,\ldots,c_n,\underbrace{0,\ldots,0}_{s-t},1,\underbrace{0,\ldots,0}_{s},1,
  \underbrace{0,\ldots,0}_{s},1,\ldots.
\end{equation*}
where $t$ is the largest integer with $0\le t\le n$ for which
$c_{n-t+1},c_{n-t+2},\ldots,c_n$ are all zero, and observe that 
there is some $j$ satisfying
\begin{equation*}
  \label{eq:373}
  n+1\le j\le n+s-t+N+1
\end{equation*}
such that $b_j\le a_j-3$ and $b_i=a_i-2$ for all $i$ with $n+1\le i\le j-1$.
It follows that
\begin{equation}
  \label{eq:376}
  \underline A\le
  \sum^n_{k=1}c_kD_k+\sum^{j-1}_{k=n+1}(a_k-2)D_k+(a_j-3)D_j+D_j.
\end{equation}
Since $j\le n+2N$,
\begin{equation*}
  \label{eq:377}
  \underline A\le\sum^n_{k=1}c_kD_k+\sum^{n+2N}_{k=n+1}(a_k-2)D_k.
\end{equation*}
We know
\begin{equation*}
  \label{eq:378}
  \sum^{n+2N}_{k=n+1}(a_k-2)D_k
  =D_n-D_{n+1}-\sum^\infty_{k=n+2N+1}(a_k-2)D_k
\end{equation*}
and since the definition of $N$ implies there is some $k$ with
$n+2N<k\le n+3N$ such that $a_k\ge3$ we conclude that
\eqref{eq:376} does not exceed $D_n-D_{n+1}-D_{n+3N}$.
The truth of the first statement of the lemma is now evident.

Now we redefine  $\beta$ as 
\begin{equation*}
  \label{eq:379}
  \beta=S(\bc_n)+D_n-D_{n+1}
\end{equation*}
and  observe that  it has Davenport coefficients 
\begin{equation*}
  \label{eq:380}
  c_1,c_2,\ldots,c_n,a_{n+1}-2,a_{n+2}-2,a_{n+3}-2,\ldots.
\end{equation*}
This contains no block $b_i,b_{i+1},\ldots,b_j$ of the form
\eqref{eq:69} nor a block $b_i,b_{i+1},\ldots,b_{i+s}$ of the form
\eqref{eq:225} and so $\beta\in E(\alpha,s)$.

Now suppose $n=0$ or $B(c_1,c_2,\ldots,c_n-1)$ is non-empty.
If $B(c_1,c_2,\ldots,c_n-1)\ne\emptyset$ then $c_n\ge1$ and so $t=0$.
Obviously $t$ is also zero if $n=0$.
As a result,
\begin{equation*}
  \label{eq:382}
  \underline A=\sum^n_{k=1}c_kD_k+D_{n+s+1}+D_{n+2s+2}+D_{n+3s+3}+\ldots.
\end{equation*}
Because $s\ge N$ we know that
\begin{equation*}
  \label{eq:383}
  D_{n+s}\ge D_{n+s+1}+D_{n+2s+2}+D_{n+3s+3}+\ldots,
\end{equation*}
and this is enough to complete the proof. 
\end{proof}

\begin{lem}
\label{thm:oaqpa}
Let $s\ge N$ and $B(\bc_n)\ne\emptyset$.
Then
\begin{align*}
  \label{eq:384}
  \underline B(\bc_n)
  &<S(\bc_n)+D_n-D_{n+1}+D_{n+2}+D_{n+s},\\
  \overline B(\bc_n)&=S(\bc_n)+D_n.
\end{align*}
Further,
\begin{equation*}
  \label{eq:386}
  \underline B(\bc_n)
  <S(\bc_n)+D_n-D_{n+1}+D_{n+s}
\end{equation*}
whenever $n=0$ or $c_n\ne a_n-2$ and $A(\bc_n)$ is non-empty.
\end{lem}

\begin{proof}
As usual, we write $B=B(\bc_n)$ and 
observe that $B$ contains the number $\beta$ whose Davenport coefficients
$b_1,b_2,b_3,\ldots$ are equal to
\begin{equation*}
  \label{eq:387}
  c_1,c_2,\ldots,c_n,a_{n+1}-1,\underbrace{0,\ldots,0}_s,1
  ,\underbrace{0,\ldots,0}_s,1,\underbrace{0,\ldots,0}_s,1,\ldots.
\end{equation*}
Therefore
\begin{equation*}
  \label{eq:394}
  \underline B\le\sum^n_{k=1}c_kD_k+(a_{n+1}-1)D_{n+1}
  +D_{n+s+2}+D_{n+2s+3}+D_{n+3s+4}+\ldots.
\end{equation*}
We note that
\begin{equation*}
  \label{eq:395}
  D_{n+s}\ge D_{n+s+2}+D_{n+2s+3}+D_{n+3s+4}+\ldots.
\end{equation*}
The first inequality of the lemma then follows since
$a_{n+1}D_{n+1}=D_n+D_{n+2}$.

For the second statement of the lemma we consider $\overline B=\beta$ where
$\beta$ is 
\begin{equation}
  \label{eq:396}
  \beta=S(\bc_n)+D_n,
\end{equation}
so  that $\beta=\sum^\infty_{k=1}c_kD_k$ where
$b_1,b_2,b_3,\ldots$ is the sequence
\begin{equation*}
  \label{eq:397}
  c_1,c_2,\ldots,c_n,a_{n+1}-1,a_{n+2}-2,a_{n+3}-2,a_{n+4}-2,\ldots,
\end{equation*}
and the rest is clear. 

Now suppose either $n=0$ or $A(\bc_n)\ne\emptyset$ and $c_n\ne a_n-2$.
In this case $\underline B$ is the number $\beta$ whose
Davenport coefficient sequence  $(b_i)$  begins with
\begin{equation}
  \label{eq:398}
  c_1,c_2,\ldots,c_n,a_{n+1}-2,a_{n+2}-2,\ldots,a_{n+s-1}-2,a_{n+s}-1
\end{equation}
and continues with
\begin{equation}
  \label{eq:399}
  \underbrace{0,\ldots,0}_s,1,\underbrace{0,\ldots,0}_s,1,
  \underbrace{0,\ldots,0}_s,1,\ldots.
\end{equation}
Clearly $(b_i)$ is a sequence of  Davenport coefficients 
and $\beta\in E(\alpha,s)$.
Further, since $(b_i)$ begins with \eqref{eq:398},
\begin{equation*}
  \label{eq:401}
  \beta\ge\sum^n_{k=1}c_kD_k+\sum^{n+s-1}_{k=n+1}(a_k-2)D_k+(a_{n+s}-1)D_{n+s}.
\end{equation*}

As the sequence of Davenport coefficients  of $\underline B$
begins with \eqref{eq:398} and continues with \eqref{eq:399},
\begin{equation*}
  \label{eq:404}
  \underline B=\sum^n_{k=1}c_kD_k+\sum^{n+s}_{k=n+1}(a_k-2)D_k+D_{n+s}
  +D_{n+2s+1}+D_{n+3s+2}+\ldots.
\end{equation*}
By using the appropriate identities of Section~\ref{sec:3gfd} we obtain
\begin{equation*}
  \label{eq:405}
  \underline B=\sum^n_{k=1}c_kD_k+D_n-D_{n+1}+D_{n+s+1}
  +D_{n+2s+1}+D_{n+3s+2}+D_{n+4s+3}+\ldots.
\end{equation*}
The usual arguments yield
\begin{equation*}
  \label{eq:406}
  D_{n+s}\ge D_{n+s+1}+D_{n+2s+1}+D_{n+3s+2}+D_{n+4s+3}+\ldots
\end{equation*}
and the truth of the final statement of the lemma is clear.
\end{proof}

\subsection{Application of Hall's Theorem}

As mentioned in the introduction to the last section, we shall now use a
theorem of Hall, namely Theorem~2.2 in
\cite{m.47:_sum_and_produc_of_contin_fract}, to show that if $s$ is large
enough then the product of the sets $E(\alpha,s)$ and $F(\alpha,s)$
contains an interval. This idea was used in the context of
inhomogeneous diophantine approximation by Cusick, Moran and Pollington, see
\cite{cusick96:_halls_ray_in_inhom_dioph_approx}.  The actual statement of
Hall's theorem~\cite{m.47:_sum_and_produc_of_contin_fract} concerns the sum of
Cantor sets but, as Hall points out, his result can be applied to products by
taking logarithms.  Specifically, we have
\begin{equation*}
  \label{eq:328}
  \log(E(\alpha,s).F(\alpha,s))=\log E(\alpha,s)\;+\;\log F(\alpha,s)
\end{equation*}
and since the logarithm function is continuous and
strictly increasing, it maps the Cantor dissections of
$G(\alpha,s)$ and $H(\alpha,s)$ to Cantor dissections of
$\log G(\alpha,s)$ and $\log H(\alpha,s)$, respectively.

Before  applying Hall's theorem we need to check that his Condition~1
holds. This condition states that if, in going from level $n$ to $n+1$, an interval
$C$ is replaced by two disjoint intervals $C_{1}$ and $C_{2}$ with an open
interval $C_{12}$ between them, so that $C_{1}\cup C_{12}\cup C_{2}=C$, then
the length of $C_{12}$ should not exceed the minimum of $|C_{1}|$ and
$|C_{2}|$. We note, as Hall does in his discussion of bounded
continued fractions, that the transition from the $n$th
to the $(n+1)$th stage of the Cantor dissections leading to the sets
$F(\alpha,s)$ and $E(\alpha,s)$ can be done by iteratively removing
just one ``middle'' interval at a time. To verify Condition~1 of Hall,
it is enough to show that for any pair of adjacent intervals formed at
the $n$th stage of the Cantor dissection to produce either
$F(\alpha,s)$  or $E(\alpha,s)$,  the minimum of their lengths exceeds
the length of the removed interval. 

We can now verify this for the Cantor dissection for $\log F(\alpha,s)$.
\begin{lem}
\label{thm:abdoas}
There is an integer $s_0\ge N$ such that if $s\ge s_0$ and if
$C_1$ and $C_2$ are non-empty neighbouring intervals arising at the
same stage of the Cantor dissection for $F(\alpha,s)$
then
\begin{equation}
  \label{eq:349}
  |\log C_{12}|\le\min\{|\log C_1|,|\log C_2|\}
\end{equation}
where $C_{12}$ is the open interval lying between $C_1$ and $C_2$.
\end{lem}

\begin{proof}
Let $s\ge N$ and let
$C_1$ and $C_2$ and $C_{12}$ be as described.
We assume without loss of generality  that $C_1$ lies to the left of $C_2$.
Our aim is to show that if $s$ is large enough then
the number
\begin{equation*}
  \label{eq:350}
  |\log C_{12}|=\log\underline{C_2}-\log\overline{C_1}
\end{equation*}
is less than or equal to both
\begin{equation*}
  \label{eq:351}
  |\log C_1|=\log\overline{C_1}-\log\underline{C_1}\text{ and }
  |\log C_2|=\log\overline{C_2}-\log\underline{C_2}.
\end{equation*}
By rearranging and using the properties of logarithms
we reduce this statement to
\begin{equation}
  \label{eq:352}
  \underline{C_1}\;\underline{C_2}\le\overline{C_1}\;\overline{C_1}\qquad\text{and}\qquad
  \underline{C_2}\;\underline{C_2}\le\overline{C_1}\;\overline{C_2}.
\end{equation}
Note that, since
\begin{equation*}
  \label{eq:353}
  4\underline{C_1}\underline{C_2}=(\underline{C_1}+\underline{C_2})^2-(\underline{C_1}-\underline{C_2})^2,
\end{equation*}
to prove the first of the inequalities in \eqref{eq:352} it is enough to show
\begin{equation*}
  \label{eq:354}
  \underline{C_1}+\underline{C_2}<2\;\overline{C_1}, 
\end{equation*}
and we concentrate on this.

Since  $C_1$ and $C_2$ arise
at the same stage of the dissection and $C_1$ lies to the left of $C_2$
we  write
\begin{equation*}
  \label{eq:355}
  C_1=C(\bc_n)\qquad\text{and}\qquad
  C_2=C(c_1,c_2,\ldots,c_{n-1},c'_n)
\end{equation*}
where $c'_n>c_n$. The key fact here is that $C(c_{1},c_{2},\ldots,c_{n},c)$ is
empty only for the extreme values of $c$, because of the conditions that describe $F(\alpha,s)$. Hence $c'_n=c_n+1$.

We write
\begin{equation*}
  \label{eq:356}
  S_1=S(\bc_n)\qquad\text{and}\qquad
  S_2=S(c_1,c_2,\ldots,c_{n-1},c'_n).
\end{equation*}
Note that $S_2=S_1+D_n$.
Assume  $t$ is  the largest integer with $0\le t\le n$ such that
all of $c_{n-t+1},c_{n-t+2},\ldots,c_n$ are zero
and $u$ the unique integer with $0\le u\le n$ such that
$c_{n-u+1},c_{n-u+2},\ldots,c_n$ is equal to \eqref{eq:323}.
We denote the corresponding integers for $C_2$ by
$t'$ and $u'$, respectively.
We know $u'=0$ else $c_1,c_2,\ldots,c_{n-1},c'_n$ ends with
\begin{equation*}
  \label{eq:357}
  a_{n-u+1}-1,a_{n-u+2}-2,a_{n-u+3}-2,\ldots,a_{n-1}-2,a_n-1
\end{equation*}
implying that $C_2=\emptyset$. Similarly, $t'=0$ since $c'_n\ge1$.
Hence
\begin{equation*}
  \label{eq:358}
  \overline {C_1}>S_1+D_n-D_{n+s-N}\qquad\text{and}\qquad\underline {C_2}<S_2+D_{n+s}.
\end{equation*}
and 
\begin{equation*}
  \label{eq:359}
  \underline {C_1}<S_1+D_{n+1}+D_{n+s}\qquad\text{and}\qquad
  \overline {C_2}>S_2+D_{n+N+1}-D_{n+s-N}.
\end{equation*}
We are now ready to consider the inequalities in \eqref{eq:352}.
The inequalities above imply that
\begin{equation*}
  \label{eq:360}
  2\;\overline{C_1}-(\underline{C_1}+\underline{C_2})>S_1+2D_n-2D_{n+s-N}-S_2-D_{n+1}-2D_{n+s}.
\end{equation*}
Further, $S_2=S_1+D_n$ and $D_{n+s-N}\ge D_{n+s}$ and thus
\begin{equation*}
  \label{eq:361}
  2\;\overline{C_1}-(\underline{C_1}+\underline{C_2})>D_n-D_{n+1}-4D_{n+s-N}.
\end{equation*}
Since \eqref{eq:26} holds for all $i\ge1$ we know there is some $s_0\ge N$
such that
\begin{equation*}
  \label{eq:362}
  1-\alpha_{n+1}-4\;\alpha_{n+1}\alpha_{n+2}\ldots\alpha_{n+s-N}>0
\end{equation*}
and hence
\begin{equation*}
  \label{eq:363}
  D_n-D_{n+1}-4D_{n+s-N}>0
\end{equation*}
if $s\ge s_0$.
Note that the size of $s_0$ is independent of $n$.
It follows that if $s\ge s_0$ then
$\underline{C_1}+\underline{C_2}<2\;\overline{C_1}$ and we have
the desired result.

For the second inequality in \eqref{eq:352} we observe that
\begin{equation*}
  \label{eq:364}
  \overline{C_1}\;\overline{C_2}-\underline{C_2}\;\underline{C_2}
  >(S_1+D_n-D_{n+s-N})(S_2+D_{n+N+1}-D_{n+s-N})-(S_2+D_{n+s})^2.
\end{equation*}
Since $S_1\ge0$ and $S_2\ge D_n$ and $D_{n+s-N}\ge D_{n+s}$ we have
\begin{equation*}
  \label{eq:365}
  \overline{C_1}\;\overline{C_2}-\underline{C_2}\;\underline{C_2}
  >(D_n-D_{n+s-N})(D_n+D_{n+N+1}-D_{n+s-N})-(D_n+D_{n+s-N})^2
\end{equation*}
and hence
\begin{equation*}
  \label{eq:366}
  \overline{C_1}\;\overline{C_2}-\underline{C_2}\;\underline{C_2}
  >D_n(D_{n+N+1}-4D_{n+s-N})-D_{n+N+1}D_{n+s-N}.
\end{equation*}
Clearly $D_n>D_{n+N+1}$ and therefore
it suffices to show that if $s$ is large enough then
\begin{equation*}
  \label{eq:367}
  D_{n+N+1}-4D_{n+s-N}>D_{n+s-N}
\end{equation*}
or equivalently
\begin{equation*}
  \label{eq:368}
  1>5\;\alpha_{n+N+2}\alpha_{n+N+3}\ldots\alpha_{n+s-N}.
\end{equation*}
As above, this is an easy consequence of \eqref{eq:26}.
\end{proof}

We can now verify that Hall's Condition~1 holds for
the dissection for $\log E(\alpha,s)$.
\begin{lem}
\label{thm:dddasd}
There is an integer $s_0\ge N$ such that if $s\ge s_0$ and if
$C_1$ and $C_2$ are non-empty neighbouring intervals arising at the
same stage of the Cantor dissection 
which produces $E(\alpha,s)$ then
\begin{equation*}
  \label{eq:407}
  |\log C_{12}|\le\min\{|\log C_1|,|\log C_2|\}
\end{equation*}
where $C_{12}$ is the open interval lying between $C_1$ and $C_2$.
\end{lem}

\begin{proof}
Let $s\ge N$ and let
$C_1$ and $C_2$ and $C_{12}$ be as described.
We may assume without loss of generality  that $C_1$ lies to the left of $C_2$.
We know from proof of Lemma~\ref{thm:abdoas} that it is sufficient to prove
the inequalities
\begin{equation}
  \label{eq:408}
  \underline{C_1}\;\underline{C_2}\le\overline{C_1}\;\overline{C_1}\text{ and }
  \underline{C_2}\;\underline{C_2}\le\overline{C_1}\;\overline{C_2}
\end{equation}
hold when $s$ is large enough.
We can also make use of statement \eqref{eq:352}.

We consider two possibilities for $C_1$.
We suppose first that
\begin{equation}
  \label{eq:409}
  C_1=A(\bc_n).
\end{equation}
In this case $B(c_1,c_2,\ldots,c_n)\ne\emptyset$ and therefore
\begin{equation*}
  \label{eq:410}
  C_2=B(\bc_n).
\end{equation*}
To see this, 
we produce a number $\beta$ that belongs to
$B(\bc_n)$.
To this end we note that in the Cantor dissection of $G(\alpha,s)$
the intervals
\begin{equation*}
  \label{eq:411}
  A(c_1,c_2,\ldots,c_{n-1},a_n-2)\text{ and }
  A(c_1,c_2,\ldots,c_{n-1},a_n-1)
\end{equation*}
have no right neighbours since they
result from the dissection of
$A(c_1,c_2,\ldots,c_{n-1})$ and $B(c_1,c_2,\ldots,c_{n-1})$,
respectively.
Therefore, either $n=0$ or $c_n\le a_n-3$.
It follows from the proof of Lemma~\ref{thm:yahdd} that
$\overline C_1$ lies in $E(\alpha,s)$ and has Davenport coefficients 
\begin{equation*}
  \label{eq:412}
  c_1,c_2,\ldots,c_n,a_{n+1}-2,a_{n+2}-2,a_{n+3}-2,\ldots.
\end{equation*}
Now let $\beta=\sum^\infty_{k=1}b_kD_k$ where $b_1,b_2,b_3,\ldots$ is
the sequence
\begin{equation*}
  \label{eq:413}
  c_1,c_2,\ldots,c_n,a_{n+1}-1,a_{n+2}-2,a_{n+3}-2,a_{n+4}-2,\ldots.
\end{equation*}
It is straightforward again to check that 
$\beta\in E(\alpha,s)$.
It now follows from \eqref{eq:320} that $\beta$ belongs to an interval
of the form $A(c'_1,c'_2,\ldots,c'_n)$ or $B(c'_1,c'_2,\ldots,c'_n)$.
By observing that
\begin{equation*}
  \label{eq:414}
  \beta=S(\bc_n)+D_n
\end{equation*}
and applying the inequalities in  \eqref{eq:316},
it can be seen that the only possibility is
$\beta\in B(\bc_n)\ne\emptyset$.

We can now apply Lemmas~\ref{thm:yahdd} and ~\ref{thm:oaqpa} to $C_1$ and $C_2$.
As usual, it is convenient to write $S=S(\bc_n)$.
Lemma~\ref{thm:yahdd} implies
\begin{equation*}
  \label{eq:415}
  \underline{C_1}<S+D_n-D_{n+1}-D_{n+3N}\qquad\text{and}\qquad
  \overline{C_1}=S+D_n-D_{n+1}
\end{equation*}
and Lemma~\ref{thm:oaqpa} implies
\begin{equation*}
  \label{eq:416}
  \underline{C_2}<S+D_n-D_{n+1}+D_{n+s}\qquad\text{and}\qquad
  \overline{C_2}=S+D_n.
\end{equation*}
It follows that
\begin{equation*}
  \label{eq:417}
  2\;\overline{C_1}-(\underline{C_1}+\underline{C_2})>D_{n+3N}-D_{n+s}.
\end{equation*}
Since \eqref{eq:26} holds for all $i\ge1$ we know there is some $s_0\ge 3N+1$
such that if $s\ge s_0$ then
\begin{equation*}
  \label{eq:418}
  1>\alpha_{n+3N+1}\alpha_{n+3N+2}\ldots\alpha_{n+s}.
\end{equation*}
We emphasis that the size of $s_0$ does not depend on $n$.
For such a choice of $s_0$ we have $D_{n+3N}>D_{n+s}$ and hence
$\underline{C_1}+\underline{C_2}<2\;\overline{C_1}$ for all $s\ge s_0$.
An application of \eqref{eq:352} gives first inequality in \eqref{eq:408}
for $s\ge s_0$.

For the second inequality in \eqref{eq:408} we observe that
\begin{equation*}
  \label{eq:419}
  \overline{C_1}\;\overline{C_2}-\underline{C_2}\;\underline{C_2}
  >(S+D_n-D_{n+1})(S+D_n)-(S+D_n-D_{n+1}+D_{n+s})^2.
\end{equation*}
Since $S\ge0$ it follows that
\begin{equation*}
  \label{eq:420}
  \overline{C_1}\;\overline{C_2}-\underline{C_2}\;\underline{C_2}
  >(D_n-D_{n+1})(D_{n+1}-2D_{n+s})-D_{n+s}^2.
\end{equation*}
Therefore it suffices to show there is some $s_0$
(which does not depend on $n$) such that
\begin{equation*}
  \label{eq:421}
  D_n-D_{n+1}>D_{n+s}\qquad\text{and}\qquad D_{n+1}-2D_{n+s}>D_{n+s}
\end{equation*}
or equivalently
\begin{equation*}
  \label{eq:422}
  1>\alpha_{n+1}+\alpha_{n+1}\alpha_{n+2}\ldots\alpha_{n+s}\qquad\text{and}\qquad
  1>3\;\alpha_{n+2}\alpha_{n+3}\ldots\alpha_{n+s}
\end{equation*}
for all $s\ge s_0$.
This is easily done with the help of \eqref{eq:26}.

The other possibility for $C_1$ is that
\begin{equation*}
  \label{eq:423}
  C_1=B(\bc_n).
\end{equation*}
It is easy to see that $\beta=\sum^\infty_{k=1}b_kD_k\in
A(c_1,c_2,\ldots,c_{n-1},c_n+1)$ where $b_1,b_2,b_3,\ldots$ is the sequence
\begin{equation*}
  \label{eq:425}
  c_1,c_2,\ldots,c_{n-1},c_n+1,a_{n+1}-2,a_{n+2}-2,a_{n+3}-2,\ldots,
\end{equation*}
and so $A(c_1,c_2,\ldots,c_{n-1},c_n+1)\ne\emptyset$.
Therefore
\begin{equation*}
  \label{eq:424}
  C_2=A(c_1,c_2,\ldots,c_{n-1},c_n+1).
\end{equation*}

Again we apply Lemmas~\ref{thm:yahdd} and ~\ref{thm:oaqpa} to $C_1$ and $C_2$.
This time we write
\begin{equation*}
  \label{eq:428}
  S_1=S(\bc_n)\qquad\text{and}\qquad
  S_2=S(c_1,c_2,\ldots,c_{n-1},c_n+1).
\end{equation*}
Note that $S_2=S_1+D_n$.
Lemma~\ref{thm:oaqpa} implies
\begin{equation*}
  \label{eq:429}
  \underline{C_1}<S_1+D_n-D_{n+1}+D_{n+2}+D_{n+s}\qquad\text{and}\qquad
  \overline{C_1}=S_1+D_n
\end{equation*}
and Lemma~\ref{thm:yahdd} implies
\begin{equation*}
  \label{eq:430}
  \underline{C_2}<S_2+D_{n+s}\qquad\text{and}\qquad
  \overline{C_2}=S_2+D_n-D_{n+1}.
\end{equation*}
These combine to yield 
\begin{equation*}
  \label{eq:431}
  2\;\overline{C_1}-(\underline{C_1}+\underline{C_2})>D_{n+1}-D_{n+2}-2D_{n+s}.
\end{equation*}
Using \eqref{eq:26} we know there is some $s_0\ge1$
(which does not depend on $n$)
such that
\begin{equation*}
  \label{eq:432}
  1>\alpha_{n+2}+2\;\alpha_{n+2}\alpha_{n+3}\ldots\alpha_{n+s}
\end{equation*}
and hence $D_{n+1}>D_{n+2}+2D_{n+s}$ for all $s\ge s_0$.
As a result $\underline{C_1}+\underline{C_2}<2\;\overline{C_1}$ if $s\ge s_0$
and using \eqref{eq:352} we conclude that the first inequality
in \eqref{eq:408} holds if $s$ is large enough.

To see that the second inequality in \eqref{eq:408} is true we note that
\begin{equation*}
  \label{eq:433}
  \overline{C_1}\;\overline{C_2}-\underline{C_2}\;\underline{C_2}
  >(S_1+D_n)(S_2+D_n-D_{n+1})-(S_2+D_{n+s})^2.
\end{equation*}
Since $S_2=S_1+D_n$ and $S\ge0$ it follows that
\begin{equation*}
  \label{eq:434}
  \overline{C_1}\;\overline{C_2}-\underline{C_2}\;\underline{C_2}
  >D_n(D_n-D_{n+1}-2D_{n+s})-D_{n+s}^2.
\end{equation*}
Therefore it suffices to show there is some $s_0$
(which does not depend on $n$) such that
\begin{equation*}
  \label{eq:435}
  D_n-D_{n+1}-2D_{n+s}>D_{n+s}
\end{equation*}
or equivalently
\begin{equation*}
  \label{eq:436}
  1>\alpha_{n+1}+3\;\alpha_{n+1}\alpha_{n+2}\ldots\alpha_{n+s}
\end{equation*}
for all $s\ge s_0$.
Again this is easily done with the help of \eqref{eq:26}.
\end{proof}

These sequence of lemmas leads to the following key precursor to the
main result.
\begin{thm}
  \label{thm:7gadaa}
There is an integer $s_0\ge N$ such that if $s\ge s_0$ and $R=N/(N+1)$ and
\begin{equation*}
  \label{eq:437}
  P_1=\frac{R^{2s}}{(1-R^{(s-N)})^2}\qquad\text{and}\qquad
  P_2=1-R^{(s-N)}
\end{equation*}
then $P_2\ge P_1$ and the interval $[P_1,P_2]$ lies in
the product of the sets $E(\alpha,s)$ and $F(\alpha,s)$.
\end{thm}

\begin{proof}
The proof applies Theorem~2.2 in Hall's paper
\cite{m.47:_sum_and_produc_of_contin_fract} to the sum
\begin{equation}
  \label{eq:438}
  \log E(\alpha,s)\;+\;\log F(\alpha,s).
\end{equation}
It is appropriate to outline why this is possible.
In the last section we showed that the sets
$E(\alpha,s)$ and $F(\alpha,s)$ are the result of
Cantor dissections of the intervals $G(\alpha,s)$ and $H(\alpha,s)$.
By applying the logarithm function it follows that the sets
$\log E(\alpha,s)$ and $\log F(\alpha,s)$ are the result of
Cantor dissections of the intervals $\log G(\alpha,s)$ and $\log H(\alpha,s)$.
We know from Lemmas~\ref{thm:abdoas} and ~\ref{thm:dddasd} that
these dissections satisfy Condition~1 in Hall's paper, if $s$ is large enough.
In other words, there is some $s_0\ge N$ such that for all $s\ge s_0$
Hall's theorem applies to the sum \eqref{eq:438}.
Note that since $R<1$ we can choose $s_0$ so that we also have $P_2>P_1$.

Hall's theorem implies that the sum \eqref{eq:438} contains the interval
\begin{equation*}
  \label{eq:439}
  [\log x_2+\log y_2-2\min\{\log x_2-\log x_1,\;\log y_2-\log y_1\}
  ,\;\log x_2+\log y_2]
\end{equation*}
where
\begin{equation*}
  \label{eq:440}
  x_1=\underline G(\alpha,s)\qquad x_2=\overline G(\alpha,s)\qquad
  y_1=\underline H(\alpha,s)\qquad y_2=\overline H(\alpha,s).
\end{equation*}
It follows immediately that
the product of $E(\alpha,s)$ and $F(\alpha,s)$ contains the interval
\begin{equation*}
  \label{eq:441}
  [x_2y_2(\max\{x_1/x_2,y_1/y_2\})^2,\;x_2y_2].
\end{equation*}
To prove the lemma it suffices to show
\begin{equation}
  \label{eq:442}
  x_2y_2(\max\{x_1/x_2,y_1/y_2\})^2\le P_1\qquad\text{and}\qquad
  x_2y_2\ge P_2.
\end{equation}
To this end we observe that $\overline G(\alpha,s)=\overline B(\;)$ and
$\overline H(\alpha,s)=\overline C(\;)$ and hence Lemma~\ref{thm:ousds} and ~\ref{thm:oaqpa} imply
$x_2=1$ and $y_2>1-D_{s-N}$. Therefore $x_2y_2>1-D_{s-N}$.
We know $D_{s-N}=\alpha_1\alpha_2\ldots\alpha_{s-N}$ and
since \eqref{eq:26} holds for all $i\ge1$
it is easy to see that the second inequality in \eqref{eq:442} is true.

For the first inequality in \eqref{eq:442} we observe that
$\underline G(\alpha,s)=\underline A(\;)$ and $\underline H(\alpha,s)=\underline C(\;)$ and hence
Lemmas~\ref{thm:ousds} and ~\ref{thm:yahdd} imply $x_1<D_s$ and $y_1<D_s$.
Thus
\begin{equation*}
  \label{eq:443}
  x_1/x_2<D_s\qquad\text{and}\qquad y_1/y_2<D_s/(1-D_{s-N}).
\end{equation*}
Clearly $x_2y_2<1$ and it follows that
\begin{equation*}
  \label{eq:444}
  x_2y_2(\max\{x_1/x_2,y_1/y_2\})^2<\frac{D_s^2}{(1-D_{s-N})^2}.
\end{equation*}
The truth of the first inequality in \eqref{eq:440} can now be seen
by expressing $D_{s-N}$ and $D_s$ in terms of the numbers
$\alpha_i$ and applying \eqref{eq:26}.
\end{proof}

Finally, we return to the sets $E(\alpha^-,s)$ and $F(\alpha^+_{r+1},s)$,
where $r\ge1$.
Recall that $\alpha^-$ and $\alpha^+$ are defined by \eqref{eq:200} and
$\alpha^+_{r+1}$ by  \eqref{eq:200}.
We know from Lemma~\ref{thm:616161}
that $\alpha^-$ and $\alpha^+$ satisfy all the constraints
we have placed on $\alpha$.
Clearly the same is true of $\alpha^+_{r+1}$.
We can, therefore,  replace $E(\alpha,s)$ and $F(\alpha,s)$ in Theorem~\ref{thm:abdoas} with
$E(\alpha^-,s)$ and $F(\alpha^+_{r+1},s)$, respectively.
In this manner we obtain the following corollary.

\begin{corollary}
There is an integer $s_0\ge N$ such that if $s\ge s_0$ and $R=N/(N+1)$ and
\begin{equation*}
  \label{eq:445}
  P_1=\frac{R^{2s}}{(1-R^{(s-N)})^2}\qquad\text{and}\qquad
  P_2=1-R^{(s-N)}
\end{equation*}
then $P_2\ge P_1$ and
the product of the sets $E(\alpha^-,s)$ and $F(\alpha^+_{r+1},s)$, where $r\ge1$,
contains the interval $[P_1,P_2]$.
\end{corollary}

\subsection{The existence of Hall's ray}

In this section, we prove the existence of
a Hall's ray in the set ${\mathcal S}_+(\alpha)$ in~\eqref{eq:4}; that
is we prove Theorem~\ref{thm:1}. 

\begin{proof}[Proof of Theorem~\ref{thm:1}]
  The proof of this theorem consists of showing that the set ${\mathcal S}_+(\alpha)$
  contains a chain of intersecting intervals whose endpoints converge to zero.
  We shall construct the chain with the help of Theorem~\ref{thm:main_r_s} and the Corollary to
  Theorem~\ref{thm:7gadaa}.

Let $s_0\ge N$ be the integer mentioned in the Corollary to Theorem~\ref{thm:7gadaa}
and define $r_0$ to be the smallest integer which is greater than
or equal to $s_0L$.
Note that since $L\ge1$ we have $s_0=\lfloor r_0/L\rfloor$ where as usual
$\lfloor x\rfloor$ denotes the largest integer which is less than or
equal to $x$.
Now let $r$ be an integer with $r\ge r_0$ and put $s=\lfloor r/L\rfloor$.
Since $r/s\ge L$ we can apply Theorem~\ref{thm:main_r_s}.
Thus for every number $x$ in the product of the sets
$E(\alpha^-,s)$ and $F(\alpha^+_{r+1},s)$ there is some $\beta$ with
$0<\beta<1$ such that
\begin{equation*}
  \label{eq:450}
  {\mathcal M}^{+}(\alpha,\beta)=\frac{xD^+_r}{1-\alpha^-\alpha^+}.
\end{equation*}
Because $r\ge r_0$ we know that $s\ge s_0$.
Therefore Theorem~\ref{thm:7gadaa} implies $P_1\le P_2$ and
the product of the sets $E(\alpha^-,s)$ and $F(\alpha^+_{r+1},s)$ contains
the interval $[P_1,P_2]$ where
\begin{equation*}
  \label{eq:451}
  P_1=\frac{R^{2s}}{(1-R^{(s-N)})^2}\qquad\text{and}\qquad
  P_2=1-R^{(s-N)}
\end{equation*}
and $R=N/(N+1)$.
It follows that for every number $\mu$ in the interval
\begin{equation}
  \label{eq:452}
  \left[\frac{P_1D^+_r}{1-\alpha^-\alpha^+}\;
  ,\;\frac{P_2D^+_r}{1-\alpha^-\alpha^+}\right]
\end{equation}
there is some $\beta$ with $0<\beta<1$ such that $M(\alpha,\beta)=\mu$.
In other words the interval \eqref{eq:452} lies in the set ${\mathcal S}^{+}(\alpha)$.
Since $r$ was any integer with $r\ge r_0$ we conclude that
${\mathcal S}^{+}(\alpha)$  contains a chain of intervals.

By choosing $s_0$ large enough
we can ensure that the intervals just mentioned intersect.
To this end let $s'=\lfloor (r+1)/L\rfloor$ and set
\begin{equation*}
  \label{eq:453}
  P'_1=\frac{R^{2s'}}{(1-R^{(s'-N)})^2}\qquad\text{and}\qquad
  P'_2=1-R^{(s'-N)}.
\end{equation*}
Note that $s'\ge s$.
According to the argument above, the interval for the integer $r+1$ is
\begin{equation*}
  \label{eq:454}
  \left[\frac{P'_1D^+_{r+1}}{1-\alpha^-\alpha^+}\;
  ,\;\frac{P'_2D^+_{r+1}}{1-\alpha^-\alpha^+}\right].
\end{equation*}
It will overlap the interval \eqref{eq:452} if both the inequalities
\begin{equation}
  \label{eq:455}
  \frac{P'_1D^+_{r+1}}{1-\alpha^-\alpha^+}\;
    \le\;\frac{P_2D^+_r}{1-\alpha^-\alpha^+}\qquad\text{and}\qquad
  \frac{P_1D^+_r}{1-\alpha^-\alpha^+}\;
    \le\;\frac{P'_2D^+_{r+1}}{1-\alpha^-\alpha^+}
\end{equation}
hold.
These inequalities become
$P'_1\alpha^+_{r+1}\le P_2$ and $P_1\le P'_2\alpha^+_{r+1}$ and,
substituting for $P_1$, $P_2$, $P'_1$ and $P'_2$ and
rearranging, we have
\begin{equation}
  \label{eq:456}
  R^{2s'}\alpha^+_{r+1}\le(1-R^{(s'-N)})^2(1-R^{(s-N)})
\end{equation}
and
\begin{equation}
  \label{eq:457}
  R^{2s}\le(1-R^{(s-N)})^2(1-R^{(s'-N)})\alpha^+_{r+1}.
\end{equation}
Now we observe that $R<1$ and hence the quantities $R^{2s}$ and $R^{(s-N)}$
and $R^{2s'}$ and $R^{(s'-N)}$ all converge to zero as
$s$ and $s'$ increase to infinity.
Since $s'\ge s\ge s_0$ and
the term $\alpha^+_{r+1}$ satisfies $1/M<\alpha^+_{r+1}<N/(N+1)$,
it is clear that by choosing $s_0$ sufficiently large we can ensure that
\eqref{eq:455} always holds.
We conclude as indicated that $s_0$ can be chosen so that
successive members in the chain of intervals in ${\mathcal S}^{+}(\alpha)$
intersect one another.
Evidently the endpoints of the interval \eqref{eq:452}. 
converge to zero as $r$ increases to infinity.
\end{proof}

\bibliographystyle{plain}
\bibliography{inhomog_bib}
\end{document}